\documentclass[12pt,a4paper]{article}

\usepackage{lmodern}

\usepackage{amsmath, amssymb, amsthm}
\usepackage{graphicx}
\usepackage[height=25.5cm, width=16cm]{geometry}
\usepackage{mathrsfs, bbm}
\PassOptionsToPackage{hyphens}{url}\usepackage{hyperref}

\usepackage[numbers]{natbib}

\usepackage[T1]{fontenc}
\usepackage[utf8]{inputenc}

\usepackage{mathtools}
\usepackage{enumitem}
\usepackage{array}
\usepackage{subfig}
\usepackage{hyphenat}
\newcommand{\mitem}{\item\(\displaystyle} 
\DeclareMathOperator{\ind}{index}

\newcommand{\diff}{d}    
\newcommand{\defeq}{\mathrel{\vcentcolon=}}
\DeclareMathOperator{\Vor}{Vor}
\providecommand\given{} 
\newcommand\SetSymbol[1][]{
   \nonscript\,#1\vert \allowbreak \nonscript\,\mathopen{}}
\DeclarePairedDelimiterX\Set[1]{\lbrace}{\rbrace}
 { \renewcommand\given{\SetSymbol[\delimsize]} #1 }
\newcommand{\Secref}[1]{Sec.~\ref{#1}}

\newcommand{\Figref}[1]{Fig.~\ref{#1}}

\DeclareMathOperator{\supp}{supp}
\newcommand{\hbit}{\hspace*{1.5pt}}

\newcommand{\EE}{\mathbb{E}}
\newcommand{\NN}{\mathbb{N}}
\newcommand{\RR}{\mathbb{R}}
\newcommand{\ZZ}{\mathbb{Z}}
\newcommand{\mcb}{\mathcal{B}}

\newcommand{\mcx}{\mathcal{X}}
\newcommand{\mcy}{\mathcal{Y}}
\newcommand{\mfc}{\mathfrak{C}}
\newcommand{\abs}[1]{|#1|}
\newcommand{\bigabs}[1]{\bigl|#1\bigr|}
\newcommand{\norm}[1]{\|#1\|}
\newcommand{\Leb}{\mathrm{Leb}}

\newcommand{\tmu}{\tilde{\mu}}
\newcommand{\tnu}{\tilde{\nu}}
\newcommand{\one}{\mathbbm{1}}
\newcommand{\eps}{\varepsilon}

\newtheorem{theorem}{Theorem} 

\newtheorem{lemma}{Lemma}
\newtheorem{remark}{Remark}

\begin{document}

\title{Semi-discrete optimal transport --- the case $p=1$}

\author{\parbox[t]{54mm}{Valentin Hartmann\footnote{Work partially supported by DFG RTG 2088.}\\ EPFL and\\ University of Goettingen}\hspace*{12mm}\parbox[t]{54mm}{Dominic Schuhmacher\\ University of Goettingen}\\[-2mm]\hspace*{2mm}}

\date{October 5, 2018}
\maketitle

\begin{abstract}
We consider the problem of finding an optimal transport plan between an absolutely continuous measure $\mu$ on $\mcx \subset \RR^d$ and a finitely supported measure $\nu$ on $\RR^d$ when the transport cost is the Euclidean distance. We may think of this problem as closest distance allocation of some ressource continuously distributed over space to a finite number of processing sites with capacity constraints.

This article gives a detailed discussion of the problem, including a comparison with the much better studied case of squared Euclidean cost (``the case $p=2$'').
We present an algorithm for computing the optimal transport plan, which is similar to the approach for $p=2$ by Aurenhammer, Hoffmann and Aronov [Algorithmica 20, 61--76, 1998] and M\'erigot~[Computer Graphics Forum 30, 1583--1592, 2011]. We show the necessary results to make the approach work for the Euclidean cost, evaluate its performance on a set of test cases, and give a number of applications. The later include goodness-of-fit partitions, a novel visual tool for assessing whether a finite sample is consistent with a posited probability density.
\vspace*{3mm}

\noindent
\textbf{MSC 2010} \, Primary 65D18; Secondary 51N20, 62-09.
\vspace*{1.5mm}

\noindent
\textbf{Key words and phrases:} Monge--Kantorovich problem;
             spatial ressource allocation;
             Wasserstein metric;
             weighted Voronoi tesselation.
\end{abstract}

\section{Introduction}
\label{sec:intro}

Optimal transport and Wasserstein metrics are nowadays among the major tools for analyzing complex data. Theoretical advances in the last decades characterize existence, uniqueness, representation and smoothness properties of optimal transport plans in a variety of different settings. Recent algorithmic advances \cite{PeyreCuturi2018} make it possible to compute exact transport plans and Wasserstein distances between discrete measures on regular grids of tens of thousands of support points, see e.g.~\cite[Section 6]{Schmitzer2016}, and to approximate such distances (to some extent) on larger and/or irregular structures, see \cite{AltschulerEtAl2017} and references therein. The development of new methodology for data analysis based on optimal transport is a booming research topic in statistics and machine learning, see e.g.~\cite{SommerfeldMunk2018,SchmitzEtAl2018,ArjovskyEtAl2017,GenevayEtAl2018,FlamaryEtAl2018}. Applications are abundant throughout all of the applied sciences, including biomedical sciences (e.g.\ microscopy or tomography images \cite{MicroscopyImaging2014,NeuroImaging2015}), geography (e.g.\ remote sensing \cite{RemoteSensing2016,RemoteSensing2017}), and computer science (e.g.\ image processing and computer graphics \cite{Papadakis2016,SolomonEtAl2015}). In brief: whenever data of a sufficiently complex structure that can be thought of as a mass distribution is available, optimal transport offers an effective, intuitively reasonable and robust tool for analysis.

More formally, for measures $\mu$ and $\nu$ on $\RR^d$ with $\mu(\RR^d)=\nu(\RR^d) < \infty$ the \emph{Wasserstein distance} of order $p \geq 1$ is defined as 
\begin{equation} \label{eq:mk}
  W_p(\mu,\nu) = \biggl( \min_{\pi} \int_{\RR^d \times \RR^d} \norm{x-y}^p \; \pi(dx,dy) \biggr)^{1/p},
\end{equation}
where the minimum is taken over all \emph{transport plans (couplings)} $\pi$ between $\mu$ and $\nu$, i.e.\ measures $\pi$ on $\RR^d \times \RR^d$ with marginals
\begin{equation*}
  \pi(A \times \RR^d) = \mu(A) \quad \text{and} \quad \pi(\RR^d \times A) = \nu(A)
\end{equation*}
for every Borel set $A \subset \RR^d$. The minimum exists by \cite[Theorem~4.1]{Villani2009} and it is readily verified, see e.g.\ \cite[after Example~6.3]{Villani2009}, that the map $W_p$ is a $[0,\infty]$-valued metric on the space of measures with fixed finite mass. The constraint linear minimization problem~\eqref{eq:mk} is known as \emph{Monge--Kantorovich problem} \cite{Kantorovich1942,Villani2009}. From an intuitive point of view, a minimizing $\pi$ describes how the mass of $\mu$ is to be associated with the mass of $\nu$ in order to make the overall transport cost minimal. 

A \emph{transport map} from $\mu$ to $\nu$ is a measurable map $T \colon \RR^d \to \RR^d$ satisfying $T_{\#}\mu=\nu$, where $T_{\#}$ denotes the push-forward, i.e.\ $(T_{\#} \mu) (A) = \mu(T^{-1}(A))$ for every Borel set $A \subset \RR^d$. We say that $T$ \emph{induces} the coupling $\pi=\pi_T$ if 
\begin{equation*}
  \pi_T(A \times B) = \mu(A \cap T^{-1}(B))
\end{equation*}
for all Borel sets $A,B \subset \RR^d$, and call the coupling $\pi$ \emph{deterministic} in that case. It is easily seen that the support of $\pi_T$ is contained in the graph of $T$. Intuitively speaking, we associate with each location in the domain of the measure $\mu$ exactly one location in the domain of the measure $\nu$ to which positive mass is moved, i.e.\ the mass of $\mu$ is not split.

The generally more difficult (non-linear) problem of finding (the $p$-th root of) 
\begin{equation}
  \inf_{T} \int_{\RR^d} \norm{x-T(x)}^p \; \mu(dx) = \inf_{T} \int_{\RR^d \times \RR^d} \norm{x-y}^p \; \pi_T(dx,dy),
\end{equation}
where the infima are taken over all transport maps $T$ from $\mu$ to $\nu$ (and are in general not attained) is known as \emph{Monge's problem} \cite{Monge1781,Villani2009}.

In practical applications, based on discrete measurement and/or storage procedures, we often face discrete measures $\mu = \sum_{i=1}^m \mu_i \delta_{x_i}$ and $\nu = \sum_{j=1}^n \nu_j \delta_{y_j}$, where $\{x_1,\ldots,x_m\}$, $\{y_1,\ldots,y_n\}$ are finite collections of support points, e.g.\ grids of pixel centers in a grayscale image. The Monge--Kantorovich problem~\eqref{eq:mk} is then simply the discrete transport problem from classical linear programming \cite{LuenbergerYe2008}:
\begin{equation} \label{eq:mkdiscr}
  W_p(\mu,\nu) = \biggl(\min_{(\pi_{ij})} \,\sum_{i=1}^m \sum_{j=1}^n d_{ij} \pi_{ij} \biggr)^{1/p},
\end{equation}
where $d_{ij} = \norm{x_i-y_j}^p$ and any measure $\pi = \sum_{i=1}^m \sum_{j=1}^n \pi_{ij} \delta_{(x_i,y_j)}$ is represented by the $m \times n$ matrix $(\pi_{ij})_{i,j}$ with nonnegative entries $\pi_{ij}$ satisfying
\begin{equation*}
  \sum_{j=1}^n \pi_{ij} = \mu_i \text{ for $1 \leq i \leq m$} \quad \text{and} \quad \sum_{i=1}^m \pi_{ij} = \nu_j \text{ for $1 \leq j \leq n$}.
\end{equation*}

Due to the sheer size of $m$ and $n$ in typical applications this is still computationally a very challenging problem; we have e.g.\ $m=n=10^6$ for $1000 \times 1000$ grayscale images, which is far beyond the performance of a standard transportation simplex or primal-dual algorithm. Recently many dedicated algorithms have been developed, such as \cite{Schmitzer2016}, which can give enormous speed-ups mainly if $p=2$ 
and can compute exact solutions for discrete transportation problems with $10^5$ support points in seconds to a few minutes, but still cannot deal with $10^6$ or more points. Approximative solutions can be computed for this order of magnitude and $p=2$ by variants of the celebrated Sinkhorn algorithm \cite{Cuturi2013,Schmitzer2016entropy,AltschulerEtAl2017}, but it has been observed that these approximations have their limitations \cite{Schmitzer2016entropy,KlattMunk2018}.

The main advantage of using $p=2$ is that we can decompose the cost function as $\norm{x-y}^2 = \norm{x}^2 + \norm{y}^2 - 2x^\top y$ and hence formulate the Monge--Kantorovich problem equivalently as
$\max_{\pi} \int_{\RR^d \times \RR^d} x^{\top} y \; \pi(dx,dy)$. For the discrete problem~\eqref{eq:mkdiscr} this decomposition is used in \cite{Schmitzer2016} to construct particularly simple so-called shielding neighborhoods. But also if one or both of $\mu$ and $\nu$ are assumed absolutely continuous with respect to Lebesgue measure, this decomposition for $p=2$ has clear computational advantages. For example if the measures $\mu$ and $\nu$ are assumed to have densities $f$ and $g$, respectively, the celebrated Brenier's theorem, which yields an optimal transport map that is the gradient of a
convex function $u$ \cite{McCann1995}, allows to solve Monge's problem by finding a numerical solution $u$ to the Monge-Amp{\`e}re equation $\det(D^2 u(x)) = f(x) \big/ g(\nabla u(x))$; see \cite[Section~6.3]{Santambrogio2015} and the references given there.

In the rest of this article we focus on the semi-discrete setting, meaning here that the measure $\mu$ is absolutely continuous with respect to Lebesgue measure and the measure $\nu$ has finite support. This terminology was recently used in \cite{Wolansky2015}, \cite{KitagawaEtAl2017}, \cite{GenevayEtAl2016} and \cite{BourneEtAl2018} among others. In the semi-discrete setting we can represent a solution to Monge's problem as a partition of $\RR^d$, where each cell is the pre-image of a support point of $\nu$ under the optimal transport map. We refer to such a partition as \emph{optimal transport partition}.

In the case $p=2$ this setting is well studied. It was shown in \cite{AHA1998} that an optimal transport partition always exists, is essentially unique, and takes the form of a Laguerre tessellation, a.k.a.\ power diagram. The authors proved further that the right tessellation can be found numerically by solving a (typically high dimensional) unconstrained convex optimization problem. Since Laguerre tessellations are composed of convex polytopes, the evaluation of the objective function can be done very precisely and efficiently. \cite{Merigot2011} elaborates details of this algorithm and combines it with a powerful multiscale idea. In \cite{KitagawaEtAl2017} a damped Newton algorithm is presented for the same objective function and the authors are able to show convergence with optimal rates.

In this article we present the corresponding theory for the case $p=1$. It is known from \cite{GeissKlein2013}, which treats much more general cost functions, that an optimal transport partition always exists, is essentially unique and takes the form of a weighted Voronoi tessellation, or more precisely an Apollonius diagram. We extend this result somewhat within the case $p=1$ in Theorems~\ref{thm:eu} and~\ref{thm:u} below. We prove then in Theorem~\ref{thm:phi} that the right tessellation can be found by optimizing an objective function corresponding to that from the case $p=2$. Since the cell boundaries in an Apollonius diagram in 2d are segments of hyperbolas, computations are more involved and we use a new strategy for computing integrals over cells and for performing line search in the optimization method. Details of the algorithm are given in Section~\ref{sec:algo} and the complete implementation can be downloaded from Github\footnote{\url{https://github.com/valentin-hartmann-research/semi-discrete-transport}} and will be included in the next version of the \texttt{transport}-package \cite{transport} for the statistical computing environment {\sf R} \cite{R}. Up to Section~\ref{sec:algo} the present paper is a condensed version of the thesis~\cite{Hartmann2016}, to which we refer from time to time for more details. In the remainder we evaluate the performance of our algorithm on a set of test cases (Section~\ref{sec:performance}), give a number of applications (Section~\ref{sec:applications}), and provide a discussion and open questions for further research (Section~\ref{sec:discussion}). 

At the time of finishing the present paper, it has come to our attention that Theorem~2.1 of \cite{KitagawaEtAl2017}, which is for very general cost funtions including the Euclidean distance (although the remainder of the paper is not), has a rather large overlap with our Theorem~\ref{thm:phi}. Within the case of Euclidean cost it assumes somewhat stronger conditons than our Theorem~\ref{thm:phi}, namely a compact domain $\mcx$ and a bounded density for $\mu$. In addition the statement is somewhat weaker as it does not contain our statement~(c). We also believe that due to the simpler setting of $p=1$ our proof is accessible to a wider audience and it is more clearly visible that the additional restrictions on $\mcx$ and $\mu$ are in fact not needed.

We end this introduction by providing some first motivation for studying the semi-discrete setting for $p=1$. This will be further substantiated in the application section~\ref{sec:applications}.

\subsection{Why semi-discrete?}

The semi-discrete setting appears naturally in problems of allocating a continuously distributed resource to a finite number of sites. Suppose for example that a fast-food chain introduces a home delivery service. Based on a density map of expected orders (the ``resource''), the management would like to establish delivery zones for each branch (the ``sites''). We assume that each branch has a fixed capacity (at least in the short run), that the overall capacity matches the total number of orders (peak time scenario), and that the branches are not too densely distributed, so that the Euclidean distance is actually a reasonable approximation to the actual travel distance; see~\cite{BoscoeEtAl2012}. We take up this example in Subsection~\ref{ssec:realloc}.

An important problem that builds on resource allocation is the quantization problem: Where to position a number of sites in such a way that the resulting resource allocation cost is minimal? See \cite[Section 4]{BourneEtAl2018} for a recent discussion using incomplete transport and $p=2$.

As a further application we propose in Subsection~\ref{ssec:vistool} optimal transport partitions as a simple visual tool for investigating local deviations from a continuous probability distribution based on a finite sample.

Since the computation of the semi-discrete optimal transport is linear in the resolution at which we consider the continuous measure (for computational purposes), it can also be attractive to use the semi-discrete setting as an approximation of either the fully continuous setting (if $\nu$ is sufficiently simple) or the fully discrete setting (if $\mu$ has a large number of support points). This will be further discussed in Section~\ref{sec:semidisctrans}.

\subsection{Why $p=1$?}
\label{ssec:whyp1}

The following discussion highlights some of the strengths of optimal transport based on an unsquared Euclidean distance ($p=1$), especially in the semi-discrete setting, and contrasts $p=1$ with $p=2$.

From a computational point of view the case $p=2$ can often be treated more efficiently, mainly due to the earlier mentioned decomposability, leading e.g.\ to the algorithms in \cite{Schmitzer2016} in the discrete and \cite{AHA1998,Merigot2011} in the semi-discrete setting. The case $p=1$ has the advantage that the Monge--Kantorovich problem has a particularly simple dual \cite[Particular Case~5.16]{Villani2009}, which is equivalent to Beckmann's problem \cite{Beckmann1952}~\cite[Theorem~4.6]{Santambrogio2015}.
If we discretize the measures (if necessary) to a common mesh of $n$ points, the latter is an optimization problem in $n$ variables rather than the $n^2$ variables needed for the general discrete transport formulation~\eqref{eq:mkdiscr}. Algorithms that make use of this reduction have been described in~\cite{SolomonEtAl2014} (for general discrete surfaces) and in~\cite[Section~4]{SchmitzerWirth2017} (for general incomplete transport), but their performance in a standard situation, e.g.~complete optimal transport on a regular grid in $\RR^d$, remains unclear. In particular we are not aware of any performance comparisons between $p=1$ and $p=2$. 

In the present paper we do not make use of this reduction, but keep the source measure $\mu$ truly continuous except for an integral approximation that we perform for numerical purposes. We describe an algorithm for the semi-discrete problem with $p=1$ that is reasonably fast, but cannot quite reach the performance of the algorithm for $p=2$ in \cite{Merigot2011}. This is again mainly due to the nice decomposition property of the cost function for $p=2$ or, more blatantly, the fact that we minimize for $p=2$ over partitions formed by line rather than hyperbola segments.

From an intuitive point of view $p=1$ and $p=2$ have both nice interpretations and depending on the application setting either the one or the other may be more justified. The difference is between thinking in terms of transportation logistics or in terms of fluid mechanics. If $p=1$ the optimal transport plan minimizes the cumulative \emph{distance} by which mass is transported. This is (up to a factor that would not change the transport plan) the natural cost in the absence of fixed costs or any other savings on long-distance transportation. If $p=2$ the optimal transport plan is determined by a pressureless potential flow
from $\mu$ to $\nu$ as seen from the kinetic energy minimization formulation of Benamou and Brenier \cite{BenamouBrenier2000}~\cite[Chapter 7]{Villani2009}.

The different behaviors in the two cases can be illustrated by the discrete toy example in Figure~\ref{fig:circle}. Each point along the incomplete circle denotes the location of one unit of mass of $\mu$ (blue x-points) and/or $\nu$ (red o-points). The unique solution for $p=1$ moves one unit of mass from one end of the circular structur to the other. This is how we would go about carrying boxes around to get from the blue scenario to the red scenario. The unique solution for $p=2$ on the other hand is to transport each unit a tiny bit further to the next one, corresponding to a (discretized) flow along the circle. It is straightforward to adapt this toy example for the semi-discrete or the continuous setting. A more complex semidiscrete example is given in Subsection~\ref{ssec:twonormals}.

\begin{figure}[th]
  \hspace*{11mm}\includegraphics[width=0.36\textwidth]{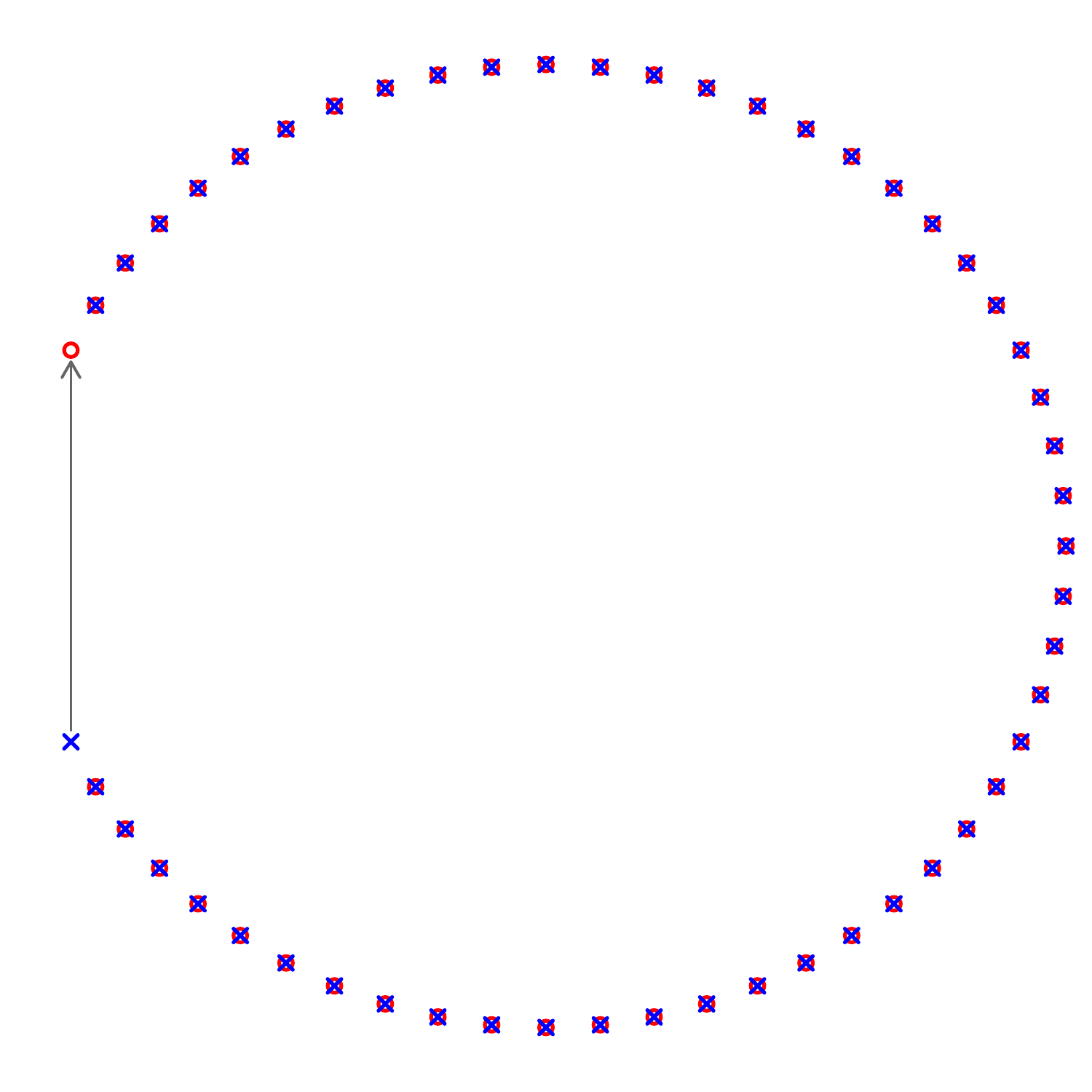}
  \hspace*{17mm}\includegraphics[width=0.36\textwidth]{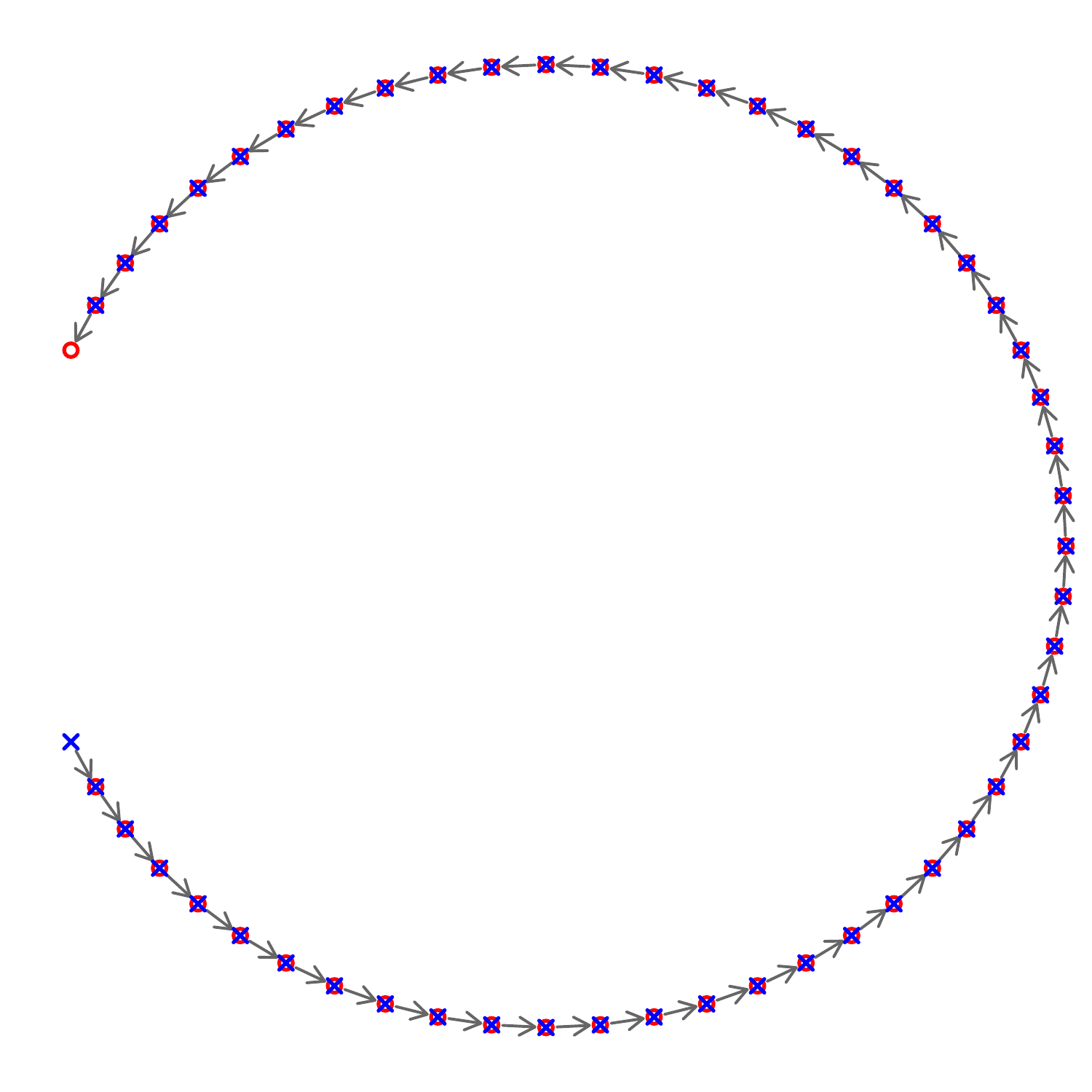}
  \vspace*{-2mm}
  
  \caption{Optimal transport maps from blue-x to red-o measure with unit mass at each point. Left: the transportation logistics solution ($p=1$); right: the fluid mechanics solution ($p=2$).}
  \label{fig:circle}
\end{figure} 

One argument in favour of the metric $W_1$ is its nice invariant properties that are not shared by the other $W_p$. In particular, considering finite measures $\mu,\nu,\alpha$ on $\RR^d$ satisfying $\mu(\RR^d) = \nu(\RR^d)$, $p \geq 1$ and $c > 0$, we have
\begin{align}
  W_1(\alpha + \mu, \alpha + \nu) &= W_1(\mu, \nu), \label{eq:add1main} \\    
  W_1(c \mu, c \nu) &= W_1(\mu, \nu).       \label{eq:scmultmain}
\end{align}
The first result is in general not true for any other $p$, the second result holds with an additional factor $c^{1/p}$ on the right hand side. We prove these statements in the appendix. These invariance properties have important implications for image analysis, where it is quite common to adjust for differring levels of brightness (in greyscale images) by affine transformations. While the above equalities show that it is safe to do so for $p=1$, it may change the resulting Wasserstein distance and the optimal transport plan dramatically for other $p$; see Appendix and Subsection~\ref{ssec:twonormals}.

It is sometimes considered problematic that optimal transport plans for $p=1$ are in general not unique. But this is not so in the semi-discrete case, as we will see in Section~\ref{sec:semidisctrans}: The minimal transport cost in~\eqref{eq:mk} is realized by a unique coupling $\pi$, which is always deterministic.
The same is true for $p=2$. A major difference in the case $p=1$ is that for $d>1$ 
each cell of the optimal transport partition contains the support point
of the target measure $\nu$ that it assigns its mass to. This can be seen as a consequence of cyclical monotonicity~\cite[beginning of~Chapter~8]{Villani2009}. In contrast, for $p=2$, optimal transport cells can be separated by many other cells from their support points, which can make the resulting partition hard to interpret without drawing corresponding arrows for the assignment; see the bottom panels of Figure~\ref{fig:transports}. For this reason we prefer to use $p=1$ for the goodness-of-fit partitions considered in Subsection~\ref{ssec:vistool}.

\section{Semi-discrete optimal transport}
\label{sec:semidisctrans}

We first concretize the semi-discrete setting and introduce some additional notation. Let now $\mcx$ and $\mcy$ be Borel subsets of $\RR^d$ and let $\mu$ and $\nu$ be \emph{probability} measures on $\mcx$ and $\mcy$, respectively. This is just for notational convenience and does not change the set of admissible measures in an essential way: We may always set $\mcx=\mcy=\RR^d$ and any statement about $\mu$ and $\nu$ we make can be easily recovered for $c\mu$ and $c\nu$ for arbitrary $c>0$.

For the rest of the article it is tacitly assumed that $d \geq 2$ to avoid certain pathologies of the one-dimensional case that would lead to a somewhat tedious distinction of cases in various results for a case that is well-understood anyway. Moreover, we always require $\mu$ to be absolutely continuous with density $\varrho$ with respect to $d$-dimensional Lebesgue measure $\Leb^d$ and to satisfy
\begin{equation} \label{eq:finite_exp}
  \int_{\mcx} \norm{x} \; \mu(dx) < \infty.
\end{equation}
We asssume further that $\nu = \sum_{j=1}^n \nu_j \delta_{y_j}$, where $n \in \NN$, $y_1, \ldots, y_n \in \mcy$ and $\nu_1, \ldots \nu_n \in (0,1]$. Condition \eqref{eq:finite_exp} guarantees that
\begin{equation} \label{eq:W1bound}
  W_1(\mu,\nu) \leq \int_{\mcx} \norm{x} \; \mu(dx) + \int_{\mcy} \norm{y} \; \nu(dy)
               =: C < \infty,
\end{equation}
which simplifies certain arguments.

The set of Borel subsets of $\mcx$ is denoted by $\mcb_{\mcx}$. Lebesgue mass is denoted by absolute value bars, i.e. $\abs{A} = \Leb^d(A)$ for every $A \in \mcb_{\mcx}$.

We call a partition $\mfc = (C_j)_{1 \leq j \leq n}$ of $\mcx$ into Borel sets satisfying $\mu(C_j) = \nu_j$ for every $j$ a \emph{transport partition} from $\mu$ to $\nu$. Any such partition characterizes a transport map $T$ from $\mu$ to $\nu$, where we set $T_{\mfc}(x) = \sum_{j=1}^n y_j 1\{x \in C_j\}$ for a given transport partition $\mfc = (C_j)_{1 \leq j \leq n}$ and $\mfc_T = (T^{-1}(y_j))_{1 \leq j \leq n}$ for a given transport map $T$.
Monge's problem for $p=1$ can then be equivalently formulated as finding
\begin{equation}
  \inf_{\mfc} \int_{\mcx} \norm{x-T_{\mfc}(x)} \; \mu(dx) = \inf_{\mfc} \sum_{j=1}^n \int_{C_j} \norm{x-y_j} \; \mu(dx),
\end{equation}
where the infima are taken over all transport partitions $\mfc = (C_j)_{1 \leq j \leq n}$ from $\mu$ to $\nu$.
Contrary to the difficulties encountered for more general measures $\mu$ and $\nu$ when considering Monge's problem with Euclidean costs, we can give a clear-cut existence and uniqueness theorem in the semi-discrete case, without any further restrictions.
\begin{theorem}
\label{thm:eu}
  In the semi-discrete setting with Euclidean costs (always including~$d\geq 2$ and \eqref{eq:finite_exp}) there is a $\mu$-a.e. unique solution $T_*$ to Monge's problem. The induced coupling $\pi_{T_*}$ is the unique solution to the Monge--Kantorovich problem, yielding
  \begin{equation}  \label{eq:w1semidisc}
    W_1(\mu,\nu) = \int_{\mcx} \norm{x-T_{*}(x)} \; \mu(dx).
  \end{equation}
\end{theorem}
\begin{proof}
The part concerning Monge's problem is a consequence of the concrete construction in Section~\ref{sec:OTviaVoronoi}; see Theorem~\ref{thm:u}.

Clearly $\pi_{T_*}$ is an admissible transport plan for the Monge--Kantorovich problem. Since $\mu$ is non-atomic and the Euclidean cost function is continuous, Theorem~B in \cite{Pratelli2007} implies that the minimum in the Monge--Kantorovich problem is equal to the infimum in the Monge problem, so $\pi_{T*}$ must be optimal.

For the uniqueness of $\pi_{T*}$ in the Monge--Kantorovich problem, let $\pi$ be an arbitrary optimal transport plan. Define the measures $\tilde{\pi}_i$ on $\mcx$ by $\tilde{\pi}_i(A)\defeq\pi(A\times \{y_i\})$ for all $A\in \mcb_{\mcx}$ and $1 \leq i \leq n$. Since $\sum_i \pi_i = \mu$, all $\pi_i$ are absolutely continuous with respect to $\Leb^d$ with densities $\tilde{\rho}_i$ satisfying $\sum\tilde{\rho}_i = \rho$. Set then $S_i\defeq\Set{x\in\mcx \given \tilde{\rho}_i>0}$.

Assume first that there exist $i,j \in \{1,\ldots,n\}$, $i \neq j$, such that $\abs{S_i\cap S_j}>0$.
Define $H_{<}^{i,j}(q)\defeq\Set{x\in\mcx \given \norm{x-y_i} < \norm{x-y_j} + q}$ and $H_{>}^{i,j}(q)$, $H_{=}^{i,j}(q)$ analogously. There exists a $q\in\RR$ for which both $S_i \cap S_j \cap H_{<}^{i,j}(q)$ and $S_i \cap S_j \cap H_{>}^{i,j}(q)$ have positive Lebesgue measure: Choose $q_1, q_2\in\RR$ such that $\abs{S_i \cap S_j \cap H_{<}^{i,j}(q_1)} > 0$ and $\abs{S_i \cap S_j \cap H_{>}^{i,j}(q_2)} > 0$; using binary search between $q_1$ and $q_2$, we find the desired $q$ in finitely many steps, because otherwise there would have to exist a $q_0$ such that $\abs{S_i \cap S_j \cap H_{=}^{i,j}(q_0)} > 0$, which is not possible. By the definition of $S_i$ and $S_j$ we thus have $\alpha = \pi_i(S_i \cap S_j \cap H_{>}^{i,j}(q)) > 0$ and $\beta = \pi_j(S_i \cap S_j \cap H_{<}^{i,j}(q)) > 0$. Switching $i$ and $j$ if necessary, we may assume $\alpha \leq \beta$. Define then
\begin{equation*}
\begin{split}
 \pi_i' &= \pi_i - \pi_i\vert_{S_i \cap S_j \cap H_{>}^{i,j}(q)} + \frac{\alpha}{\beta} \hbit \pi_j\vert_{S_i \cap S_j \cap H_{<}^{i,j}(q)}, \\
 \pi_j' &= \pi_j + \pi_i\vert_{S_i \cap S_j \cap H_{>}^{i,j}(q)} - \frac{\alpha}{\beta} \hbit \pi_j\vert_{S_i \cap S_j \cap H_{<}^{i,j}(q)}
\end{split}
\end{equation*}
and $\pi_k' = \pi_k$ for $k \not\in \{i,j\}$.
It can be checked immediately that the measure $\pi'$ given by $\pi'(A \times \{y_i\}) = \pi'_i(A)$ for all $A \in \mcb_{\mcx}$ and all $i \in \{1,2,\ldots,n\}$ is a transport plan from $\mu$ to $\nu$ again. It satisfies
\begin{equation*}
\begin{split}
 &\int_{\mcx\times\mcy} \norm{x-y} \; \pi'(dx, dy) - \int_{\mcx\times\mcy} \norm{x-y} \; \pi(dx, dy) \\
 &= \int_{S_i \cap S_j \cap H_{>}^{i,j}(q)} \bigl( -\norm{x-y_i} + \norm{x-y_j} \bigr) \; \pi_i(dx)
 {} + \frac{\alpha}{\beta} \int_{S_i \cap S_j \cap H_{<}^{i,j}(q)} \bigl( \norm{x-y_i} - \norm{x-y_j} \bigr) \; \pi_j(dx) \\
 &< 0,
\end{split}
\end{equation*}
because the integrands are strictly negative on the sets over which we integrate. But this contradicts the optimality of $\pi$.

We thus have proved that $\abs{S_i\cap S_j} = 0$ for all pairs with $i \neq j$. This implies that we can define a transport map $T$ inducing $\pi$ in the following way. If $x\in S_i\setminus(\cup_{j\neq i} S_j)$ for some $i$, set $T(x)\defeq y_i$. Since the intersections $S_i\cap S_j$ are Lebesgue null sets, the value of $T$ on them does not matter. So we can for example set $T(x)\defeq y_{1}$ or $T(x)\defeq y_{i_0}$ for $x\in \bigcap_{i \in I} S_{i} \setminus \bigcap_{i \in I^c} S_{i}$, where $I \subset \{1,\ldots,n\}$ contains at least to elements and $i_0 = \min(I)$. It follows that $\pi_T = \pi$. But by the optimality of $\pi$ and Theorem \ref{thm:u} we obtain $T=T_*$ $\mu$-almost surely, which implies $\pi = \pi_T = \pi_{T_*}$.
\end{proof}

It will be desirable to know in what way we may approximate the continuous and discrete Monge--Kantorovich problems by the semi-discrete problem we investigate here. 

In the fully continuous case, we have a measure $\tnu$ on $\mcy$ with density $\tilde{\varrho}$ with respect to $\Leb^d$ instead of the discrete measure $\nu$. In the fully discrete case, we have a discrete measure $\tmu = \sum_{i=1}^m \tmu_i \delta_{x_i}$ instead of the absolutely continuous measure $\mu$, where $m \in \NN$, $x_1, \ldots, x_m \in \mcx$ and $\tmu_1, \ldots \tmu_m \in (0,1]$. In both cases existence of an optimal transport plan is still guaranteed by \cite[Theorem~4.1]{Villani2009}, however we lose to some extent the uniqueness property.

One reason for this is that mass transported within the same line segment can be reassigned at no extra cost; see the discussion on transport rays in Section~6 of \cite{AmbrosioPratelli2003}. In the continuous case this is the only reason, and uniqueness can be restored by minimizing a secondary functional (e.g.\ total cost with respect to $p>1$) over all optimal transport plans; see Theorem~7.2 in \cite{AmbrosioPratelli2003}. 

In the discrete case uniqueness depends strongly on the geometry of the support points of $\tmu$ and $\nu$. In addition to collinearity of support points, equality of interpoint distances can also lead to non-unique solutions. While uniqueness can typically be achieved when the support points are in sufficiently general position, we are not aware of any precise result to this effect.

When approximating the continuous problem with measures $\mu$ and $\tnu$ by a semi-discrete problem, we quantize the measure $\tnu$ to obtain a discrete measure $\nu = \sum_{j=1}^n \nu_j \delta_{y_j}$, where $\nu_j = \tnu(N_j)$ for a partition $(N_j)$ of $\supp(\tnu)$. The error we commit in Wasserstein distance by discretization of $\tnu$ is bounded by the quantization error, i.e.
\begin{equation} \label{eq:quanterror}
   \bigabs{W_1(\mu,\tnu) - W_1(\mu,\nu)} \leq W_1(\tnu,\nu) \leq \sum_{j=1}^n \int_{N_j} \norm{y-y_j} \; \tnu(dy).
\end{equation}
We can compute $W_1(\tnu,\nu)$ exactly by solving another semi-discrete transport problem, using the algorithm described in Section~\ref{sec:algo} to compute an optimal partition $(N_j)$ for the second inequality above. However, choosing $\nu$ for given $n$ in such a way that $W_1(\tnu,\nu)$ is minimal is usually practically infeasible. So we would use an algorithm that makes $W_1(\tnu,\nu)$ reasonably small, such as a suitable version of Lloyd's algorithm; see Subsection~\ref{ssec:multiscale} below.

When approximating the discrete problem with measures $\tmu$ and $\nu$ by a semi-discrete problem, we blur each mass $\tmu_i$ of $\tmu$ over a neighborhood of $x_i$ using a probability density $f_i$, to obtain a measure $\mu$ with density $\varrho(x) = \sum_{i=1}^m \tmu_i f_i(x_i)$. Typical examples use $f_i(x) = \frac{1}{h^d} \varphi\bigl(\frac{x-x_i}{h}\bigr)$, where $\varphi$ is the standard normal density and the bandwidth $h>0$ is reasonably small, or $f_i(x) = \frac{1}{\abs{M_i}} \one_{M_i}(x)$, where $M_i$ is some small neighborhood of $x_i$. In practice, discrete measures are often available in the form of images, where the support points $x_i$ form a fine rectangular grid; then the latter choice of $f_i$s is very natural, where the $M_i$s are just adjacent squares, each with an $x_i$ at the center.
There are similar considerations for the approximation error as in the fully continuous case above. In particular the error we commit in Wasserstein distance is bounded by the blurring error
\begin{equation} \label{eq:blurrerror}
   \bigabs{W_1(\tmu,\nu) - W_1(\mu,\nu)} \leq W_1(\tmu,\mu) \leq \sum_{i=1}^m \tmu_i \int_{\RR^d} \norm{x-x_i} f(x) \; dx.
\end{equation}
The right hand side is typically straightforward to compute exactly, e.g.\ in the normal density and grid cases described above. It can be made small by choosing the bandwidth $h$ very small or picking sets $M_i$ of small radius $r = \sup_{x \in M_i} \norm{x-x_i}$. 

What about the approximation properties of the optimal transport plans obtained by the semi-discrete setting?
Theorem~5.20 in \cite{Villani2009} implies for $\nu^{(k)} \to \tnu$ weakly and $\mu^{(k)} \to \tmu$ weakly that every subsequence of the sequence of optimal transport plans $\pi^{(k)}_*$ between $\mu^{(k)}$ and $\nu^{(k)}$ has a further subsequence that converges weakly to an optimal transport plan $\pi_*$ between $\mu$ and $\nu$. This implies that for every $\eps>0$ there is a $k_0 \in \NN$ such that for any $k \geq k_0$ the plan $\pi^{(k)}$ is within distance $\eps$ (in any fixed metrization of the weak topology) of \emph{some} optimal transport plan between $\mu$ and $\nu$,
which is the best we could have hoped for in view of the non-uniqueness of optimal transport plans we have in general. If (in the discrete setting) there is a unique optimal transport plan $\pi_*$, this yields that $\pi_*^{(k)} \to \pi_*$ weakly.

\section{Optimal transport maps via weighted Voronoi tessellations}
\label{sec:OTviaVoronoi}

As shown for bounded $\mcx$ in \cite{GeissKlein2013}, the solution to the semi-discrete transport problem has a nice geometrical interpretation, which is similar to the well-known result in~\cite{AHA1998}: We elaborate below that the sets $C^{*}_j$ of the optimal transport partition are the cells of an additively weighted Voronoi tessellation of $\mcx$ around the support points of $\nu$.

For the finite set of points $\{y_1, \dots, y_n\}$ and a vector $w\in\RR^n$ that assigns to each $y_j$ a weight $w_j$, the \emph{additively weighted Voronoi tessellation} is the set of cells
\[\Vor_w(j) = \Set{x\in\mcx \given \norm{x-y_j} - w_j \leq \norm{x-y_k} - w_k \quad\text{for all } k \neq j}, \quad j=1, \dots, n.\]
Note that adjacent cells $\Vor_w(j)$ and $\Vor_w(k)$ have disjoint interiors. The intersection of their boundaries is a subset of $H = \Set{x\in\mcx \given \norm{x-y_j} - \norm{x-y_k} = w_j - w_k}$, which is easily seen to have Lebesgue measure (and hence $\mu$-measure) zero. If $d=2$, the set $H$ is a branch of a hyperbola with foci at $y_j$ and $y_k$. It may also be interpreted as the set of points that have the same distance from the spheres $S(y_j,w_j)$ and $S(y_k,w_k)$, where $S(y,w) = \Set{x \in \mcx \given \norm{x-y} = w}$. See Figure~\ref{fig:voronoi} for an illustration of these properties.
\begin{figure}
	\centering
	\includegraphics[width=0.8\textwidth]{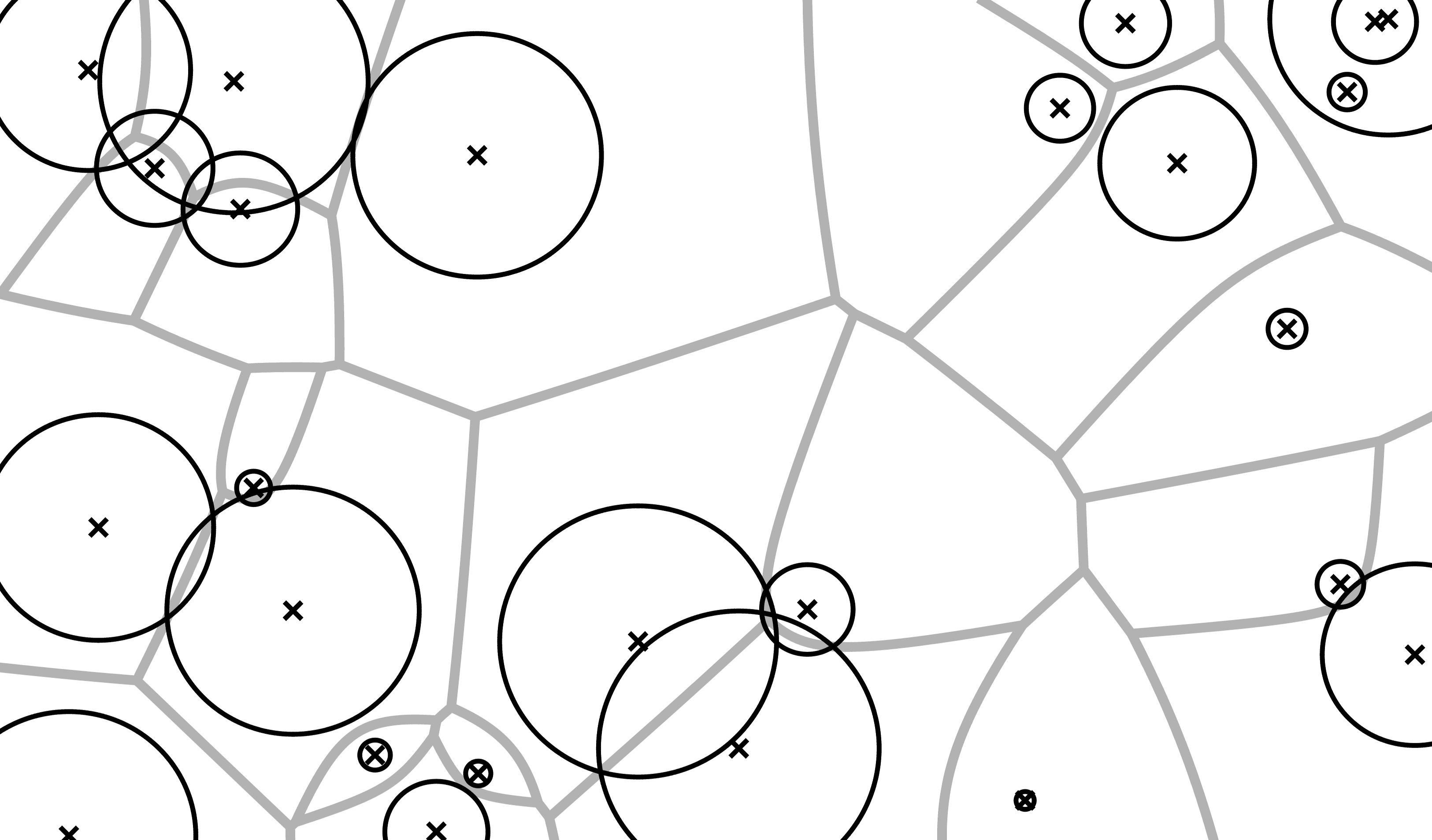}
	\caption{An additively weighted Voronoi tessellation with 25 cells.}
	\label{fig:voronoi}
\end{figure}

Of course not all weighted Voronoi tessellations are valid transport partitions from~$\mu$ to~$\nu$. But suppose we can find a weight vector $w$ such that the resulting Voronoi tessellation satisfies indeed $\mu(\Vor_w(j)) = \nu_j$ for every $j \in \{1,\ldots,n\}$; we call such a $w$ \emph{adapted} to $(\mu,\nu)$. Then this partition is automatically optimal.
\begin{theorem}
\label{thm:u}
  If $w\in\RR^n$ is adapted to $(\mu, \nu)$, then $(\Vor_w(j))_{1 \leq j \leq n}$ is the $\mu$-almost surely unique optimal transport partition from $\mu$ to $\nu$.
\end{theorem}
A proof was given in \cite[Theorem~2]{GeissKlein2013} for more general distance functions, but required $\mcx$ to be bounded. For the Euclidean distance we consider here, we can easily extend it to unbounded $\mcx$; see~\cite[Theorem~3.2]{Hartmann2016}.

Having identified this class of optimal transport partitions, it remains to show that for each pair $(\mu, \nu)$ we can find an adapted weight vector. We adapt the approach of \cite{AHA1998} to the case $p=1$, which gives us a constructive proof that forms the basis for the algorithm in Section \ref{sec:algo}. Our key tool is the function $\Phi$ defined below.
\begin{theorem}
\label{thm:phi}
	Let $\Phi : \RR^n \to \RR$,
\[\Phi(w) = \sum_{j=1}^n\left(-\nu_j w_j - \int_{\Vor_w(j)} \left(\lVert x - y_j\rVert - w_j\right) \; \mu(\diff x)\right).\]
Then
\begin{enumerate}[label=\alph*)]
	\item \(\Phi\) is convex; \label{enum:phi1}
\item \(\Phi\) is continuously differentiable with partial derivatives
\[\frac{\partial\Phi}{\partial w_j}(w) = -\nu_j+\mu(\Vor_w(j));\] \label{enum:phi2}
\item $\Phi$ takes a minimum in $\RR^n$. \label{enum:phi3}
\end{enumerate}
\end{theorem}
\begin{remark}
\label{rem:masterplan}
  Let  $w^* \in \RR^n$ be a minimizer of $\Phi$. Then by Theorem~\ref{thm:phi}\hbit b)
  \[\mu(\Vor_{w^*}(j))-\nu_j = \frac{\partial\Phi}{\partial w_j}(w^*) = 0 \quad \text{for every $j \in \{1,\ldots,n\}$},\]
  i.e.\ $w_{*}$ is adapted to $(\mu, \nu)$. Theorem~\ref{thm:u} yields that $(Vor_{w^*}(j))_{1 \leq j \leq n}$ is the $\mu$-almost surely unique optimal transport partition from $\mu$ to $\nu$.
\end{remark}
\begin{proof}[of Theorem~\ref{thm:phi}]
We take a few shortcuts; for full technical details see Chapter~3 of~\cite{Hartmann2016}.

Part~\ref{enum:phi1} relies on the observation that $\Phi$ can be written as
\[\Phi(w) = \sum_j(-\nu_j w_j) - \Psi(w) \]
where
\[\Psi(w) = \int_{\mcx} (\norm{x - T^w(x)} - w_{T^w(x)}) \; \mu(\diff x),\]
$T^w$ denotes the transport map induced by the Voronoi tessellation with weight vector~$w$ and we write $w_{y_j}$ instead of $w_j$ for convenience. By definition of the weighted Voronoi tessellation $\Psi$ is the infimum of the affine functions
\[\Psi_f \colon \RR^n \to \RR, \ w \mapsto \int_{\mcx} (\norm{x - f(x)} - w_{f(x)}) \; \mu(\diff x)\]
over all measurable maps $f$ from $\mcx$ to $\mcy$. Since pointwise infima of affine functions are concave and the first summand of $\Phi$ is linear, we see that $\Phi$ convex.

By geometric arguments it can be shown that $[w \mapsto \mu(\Vor_w(j))]$ is continuous; see~\cite[Lemma~3.3]{Hartmann2016}. A short computation involving the representation $\Psi(w) = \inf_f \Psi_f(w)$ used above yields for the difference quotient of $\Psi$, writing $e_j$ for the $j$-th standard basis vector and letting $h \neq 0$,
\[ \biggl| \frac{\Psi(w+he_j)-\Psi(w)}{h} + \mu(\Vor_w(j)) \biggr| \leq \bigl| -\mu(\Vor_{w+he_j}(j)) + \mu(\Vor_w(j)) \bigr| \longrightarrow 0 \]
as $h \to 0$. This implies that $\Psi$ is $C^1$ with (continuous) $j$-th partial derivative $-\mu(\Vor_w(j))$ and hence statement~\ref{enum:phi2} follows.

Finally, for the existence of a minimizer of $\Phi$ we consider an arbitrary sequence $(w^{(k)})_{k\in\NN}$ of weight vectors in $\RR^n$ with
\[\lim_{k\to\infty} \Phi(w^{(k)}) = \inf_{w\in\RR^n} \Phi(w).\]
We show below that a suitably shifted version of $(w^{(k)})_{k\in\NN}$ that has the same $\Phi$-values contains a bounded subsequence. This subsequence then has a further subsequence $(u^{(k)})$ which converges towards some $u \in \RR^n$. Continuity of $\Phi$ yields 
\[\Phi(u) = \lim_{k\to\infty} \Phi(u^{(k)}) = \inf_{w\in\RR^n} \Phi(w)\]
and thus statement~\ref{enum:phi3}.

To obtain the bounded subsequence, note first that adding to each weight the same constant neither affects the Voronoi tessellation nor the value of $\Phi$. We may therefore assume $w_j^{(k)} \geq 0$, $1 \leq j \leq n$, for all $k \in \NN$. Choosing an entry $i$ and an infinite set $K\subset\NN$ appropriately leaves us with a sequence $(w^{(k)})_{k\in K}$ satisfying $w_i^{(k)} \geq w_j^{(k)}$ for all $j$ and $k$. Taking a further subsequence $(w^{(l)})_{l \in \NN}$ for some infinite $L\subset K$ allows the choice of an $R \geq 0$ and the partitioning of $\{1,\dots,n\}$ into two sets $A$ and $B$ such that for every $l \in \NN$
\begin{enumerate}[label=\roman*)]
	\mitem 0 \leq w_i^{(l)} - w_j^{(l)} \leq R \quad \text{if } j \in A,\) \label{enum:subsequence1}
	\mitem w_i^{(l)} - w_j^{(l)} \geq \ind(l) \quad \text{if } j \in B,\) \label{enum:subsequence2}
\end{enumerate}
where $\ind(l)$ denotes the rank of $l$ in $L$. 

Assume that $B \neq \emptyset$. The Voronoi cells with indices in $B$ will at some point be shrunk to measure zero, meaning there exists an $N\in\NN$ such that
\[\sum_{j\in A} \mu\bigl(\Vor_{w^{(l)}}(j)\bigr) = 1 \quad\text{for all } l \geq N.\]
Write
\[ \underline{w}_A^{(l)} = \min_{i \in A} \hbit w_i^{(l)} \quad \text{and} \quad \overline{w}_B^{(l)} = \max_{i \in B} \hbit w_i^{(l)},\]
and recall the constant $C$ from \eqref{eq:W1bound}, which may clearly serve as an upper bound for the transport cost under an \emph{arbitrary} plan. We then obtain for every $l \geq N$ 
\begin{equation*}
\begin{split}
  \Phi(w^{(l)}) &= \sum_{j=1}^n \biggl( -\nu_j w_j^{(l)} - \int_{\Vor_{w^{(l)}}(j)} \bigl( \norm{x-y_j} - w_j^{(l)} \bigr) \; \mu(dx) \biggr) \\
  &\geq -C + \sum_{j=1}^n w_j^{(l)} \Bigl( \mu\bigl(\Vor_{w^{(l)}}(j)\bigr) - \nu_j \Bigr) \\
  &=  -C + \sum_{j \in A} w_j^{(l)} \Bigl( \mu\bigl(\Vor_{w^{(l)}}(j)\bigr) - \nu_j \Bigr)
         - \sum_{j \in B} w_j^{(l)} \nu_j \\
  &\geq -C-R + \underline{w}_A^{(l)} \biggl( 1 - \sum_{j \in A} \nu_j \biggr)
         - \overline{w}_B^{(l)} \sum_{j \in B} \nu_j \\
  &\geq -C-2R + \ind(l),
\end{split}
\end{equation*}
which contradicts the statement $\lim_{k\to\infty} \Phi(w^{(k)}) = \inf_{w\in\RR^n} \Phi(w) < \infty$. Thus we have $B=\emptyset$.

We can then simply turn $(w^{(l)})_{l\in L}$ into a bounded sequence by substracting the minimal entry $\underline{w}^{(l)} = \min_{1 \leq i \leq n} w_i^{(l)}$ from each $w_j^{(l)}$ for all $l \in L$.
\end{proof}

\section{The algorithm}
\label{sec:algo}

The previous section provides the theory needed to compute the optimal transport partition. It is sufficient to find a vector $w^*$ where $\Phi$ is locally optimal. By convexity, $w^*$ is then a global minimizer of $\Phi$ and Remark~\ref{rem:masterplan} identifies the $\mu$-a.e.\ unique optimal transport partition as $(\Vor_{w^*}(j))_{1 \leq j \leq n}$.

For the optimization process we can choose from a variety of methods thanks to knowing the gradient $\nabla \Phi$ of $\Phi$ analytically from Theorem~\ref{thm:phi}. We consider iterative methods that start at an initial weight vector $w^{(0)}$ and apply steps of the form
\begin{equation*}
  w^{(k+1)} = w^{(k)} + t_k \hbit \Delta w^{(k)}, \quad k \geq 0,
\end{equation*}
where $\Delta w^{(k)}$ denotes the search direction and $t_k$ the step size.

Newton's method would use $\Delta w^{(k)} = -\bigl( D^2 \Phi(w^{(k)}) \bigr)^{-1} \nabla \Phi(w^{(k)})$, but the Hessian matrix $D^2 \Phi(w^{(k)})$ is not available to us. We therefore use a quasi-Newton method that makes use of the gradient. Just like \cite{Merigot2011} for the case $p=2$, we have obtained many good results using L-BFGS \cite{Nocedal80}, the limited-memory variant of the Broyden--Fletcher--Goldfarb--Shanno algorithm, which uses the value of the gradient at the current as well as at preceding steps for approximating the Hessian. The limited-memory variant works without storing the whole Hessian of size $n\times n$, which is important since in applications our $n$ is typically large.

To determine a suitable step size $t_k$ we use the Armijo rule \cite{Armijo1966}, which has proven to be well suited for our problem. It considers different values for $t_k$ until it arrives at one that sufficiently decreases $\Phi$. We also considered replacing the Armijo rule with the strong Wolfe conditions \cite{Wolfe1969,Wolfe1971} as done in \cite{Merigot2011}, which contain an additional decrease requirement on the gradient. In our case, however, this requirement could often not be fulfilled because of the pixel splitting method used for computing the gradient (cf. \Secref{ssec:computation}), which made it less suited.

\subsection{Multiscale approach to determine starting value}
\label{ssec:multiscale}

To find a good starting value $w^{(0)}$ we use a multiscale method similar to the one proposed in~\cite{Merigot2011}. We first create a decomposition of $\nu$, i.e.\ a sequence $\nu = \nu^{(0)}, \dots, \nu^{(L)}$ of measures with decreasing cardinality of the support. Here $\nu^{(l)}$ is obtained as a coarsening of $\nu^{(l-1)}$ by merging the masses of several points into one point.

It seems intuitively reasonable to choose $\nu^{(l)}$ in such a way that $W_1(\nu^{(l)},\nu^{(l-1)})$ is as small as possible, since the latter bounds $\abs{W_1(\mu,\nu^{(l)}) - W_1(\mu,\nu^{(l-1)})}$. This comes down to a capacitated location-allocation problem, which is NP-hard even in the one-dimensional case; see~\cite{SheraliNordai1988}. Out of speed concerns and since we only need a reasonably good starting value for our algorithm, we decided to content ourselves with the same weighted $K$-means clustering algorithm used by \cite{Merigot2011} (referred to as Lloyd's algorithm), which iteratively improves an initial aggregation of the support of $\nu^{(l-1)}$ into $\abs{\supp(\nu^{(l)})}$ clusters towards local optimality with respect to the \emph{squared} Euclidean distance.
The resulting $\nu^{(l)}$ is then the discrete measure with the cluster centers as its support points and as weights the summed up weights of the points of $\nu^{(l-1)}$ contained in each cluster; see Algorithm~3 in~\cite{Hartmann2016}.
The corresponding weighted $K$-median clustering algorithm, based on alternating between assignment of points to clusters and recomputation of cluster centers as the \emph{median} of all weighted points in the cluster, should intuitively give a $\nu^{(l)}$ based on which we obtain a better starting solution. This may sometimes compensate for the much longer time needed for performing $K$-median clustering.

Having created the decomposition $\nu = \nu^{(0)}, \dots, \nu^{(L)}$, we minimize $\Phi$ along the sequence of these coarsened measures, beginning at $\nu^{(L)}$ with the initial weight vector $w^{(L,0)} = 0\in\RR^{\abs{\supp(\nu^{(L)})}}$ and computing the optimal weight vector $w^{(L,*)}$ for the transport from $\mu$ to $\nu^{(L)}$. Every time we pass to a finer measure $\nu^{(l-1)}$ from the coarser measure $\nu^{(l)}$, we generate the initial weight vector $w^{(l-1,0)}$ from the last optimal weight vector $w^{(l,*)}$ by assigning the weight of each support point of $\nu^{(l)}$ to all the support points of $\nu^{(l-1)}$ from whose merging the point of $\nu^{(l)}$ originated; see also Algorithm~2 in~\cite{Hartmann2016}.

\subsection{Numerical computation of $\Phi$ and $\nabla \Phi$}
\label{ssec:computation}

For practical computation we assume here that $\mcx$ is a bounded rectangle in $\RR^2$ and that the density of the measure $\mu$ is of the form
\begin{equation*}
  \varrho(x) = \sum_{i \in I} a_{i} \one_{Q_{i}}(x)
\end{equation*}
for $x \in \mcx$, where we assume that $I$ is a finite index set and $(Q_i)_{i \in I}$ is a partition of the domain $\mcx$ into (small) squares, called \emph{pixels}, of equal side length. This is natural if $\varrho$ is given as a grayscale image and we would then typically index the pixels $Q_i$ by their centers $i \in I \subset \ZZ^2$. It may also serve as an approximation for arbitrary $\varrho$. It is however easy enough to adapt the following considerations to more general (not necessarily congruent) tiles and to obtain better approximations if the function $\varrho$ is specified more generally than piecewise constant.  

The optimization procedure requires the non-trivial evaluation of $\Phi$ at a given weight vector $w$. This includes the integration over Voronoi cells and therefore the construction of a weighted Voronoi diagram. The latter task is solved by the package \emph{2D Apollonius Graphs} as part of the \emph{Computational Geometry Algorithms Library} \cite{CGAL2015}. The integrals we need to compute are
\[\int_{\Vor_w(j)} \rho(x) \; \diff x \quad\text{and}\quad \int_{\Vor_w(j)} \lVert x - y_j\rVert \rho(x) \; \diff x.\]
By definition the boundary of a Voronoi cell $\Vor_w(j)$ is made up of hyperbola segments, each between $y_j$ and one of the other support points of $\nu$. The integration could be performed by drawing lines from $y_j$ to the end points of those segments and integrating over the resulting triangle-shaped areas separately. This would be executed by applying an affinely-linear transformation that moves the hyperbola segment onto the hyperbola $y=1/x$ to both the area and the function we want to integrate. The required transformation can be found in \cite[Section 5.6]{Hartmann2016}.

However, we take a somewhat more crude, but also more efficient path here, because it is a quite time-consuming task to decide which pixels intersect which weighted Voronoi cells and then to compute the (areas of the) intersections. We therefore approximate the intersections by splitting the pixels into a quadratic number of subpixels (unless the former are already very small) and assuming that each of them is completely contained in the Voronoi cell in which its center lies. This reduces the problem from computing intersections to determining the corresponding cell for each center, which the data structure used for storing the Voronoi diagram enables us to do in roughly $\mathcal{O}(\log n)$ time; see \cite{KaravelasYvinec2002}. The operation can be performed even more efficiently: When considering a subpixel other than the very first one, we already know the cell that one of the neighboring subpixel's center belongs to. Hence, we can begin our search at this cell, which is either already the cell we are looking for or lies very close to it.

The downside of this approximation is that it can make the L-BFGS algorithm follow search directions along which the value of $\Phi$ cannot be sufficiently decreased even though there exist different directions that allow a decrease. This usually only happens near a minimizing weight vector and can therefore be controlled by choosing a not too strict stopping criterion for a given subpixel resolution, see the next subsection.

\subsection{Our implementation}
\label{ssec:implementation}

Implementing the algorithm described in this section requires two technical choices: The number of subpixels every pixel is being split into and the stopping criterion for the minimization of $\Phi$.
We found that choosing the number of subpixels to be the smallest square number such that their total number is larger than or equal to $1000n$ gives a good compromise between performance and precision.

The stopping criterion is implemented as follows: We terminate the optimization process once $\norm{\nabla \Phi(w)}_1/2 \leq \eps$ for some $\eps > 0$. Due to Theorem \ref{thm:phi}\ref{enum:phi2} this criterion yields an intuitive interpretation: $\norm{\nabla \Phi(w)}_1/2$ is the amount of mass that is being mistransported, i.e.\ the total amount of mass missing or being in surplus at some $\nu$-location~$y_i$ when transporting according to the current tessellation.
In our experience this mass is typically rather proportionally distributed among the different cells and tends to be assigned in a close neighbourhood of the correct cell rather than far away. So even with somewhat larger $\eps$, the computed Wasserstein distance and the overall visual impression of the optimal transport partition remain mostly the same. In the numerical examples in Sections~\ref{sec:performance} and \ref{sec:applications} we choose the value of $\eps = 0.05$.

We implemented the algorithm in C++ and make it available on GitHub\footnote{\url{https://github.com/valentin-hartmann-research/semi-discrete-transport}} under the MIT license. Our implementation uses libLBFGS \cite{libLBFGS2015} for the L-BFGS procedure and the geometry library CGAL \cite{CGAL2015} for the construction and querying of weighted Voronoi tessellations. The repository also contains a Matlab script to visualize such tessellations. Our implementation will also be included in a future version of the \texttt{transport}-package \cite{transport} for the statistical computing environment {\sf R} \cite{R}.

\section{Performance evaluation}
\label{sec:performance}

We evaluated the performance of our algorithm by randomly generating measures \(\mu\) and \(\nu\) with varying features and computing the optimal transport partitions between them. The measure \(\mu\) was generated by simulating its density $\varrho$ as a Gaussian random field with Mat{\'e}rn covariance function on the rectangle $[0,1] \times [0,0.75]$, applying a quadratic function and normalizing the result to a probability density. Corresponding images were produced at resolution \(256\times 196\) pixels and were further divided into 25 subpixels each to compute integrals over Voronoi cells. In addition to a variance parameter, which we kept fixed,
the Mat{\'e}rn covariance function has parameters for the scale $\gamma$ of the correlations, which we varied among $0.05$, $0.15$ and $0.5$, and the smoothness $s$ of the generated surface, which we varied between $0.5$ and $2.5$ corresponding to a continuous surface and a $C^2$-surface, respectively. The simulation mechanism is similar to the ones for classes 2--5 in the benchmark DOTmark proposed in~\cite{SchrieberEtAl2017}, but allows to investigate the influence of individual parameters more directly. Figure~\ref{fig:paramcombos} shows one realization for each parameter combination. For the performance evaluation we generated 10 realizations each.
\begin{figure}[ht]
  \centering
  \hspace*{1mm}\includegraphics[width=.32\textwidth]{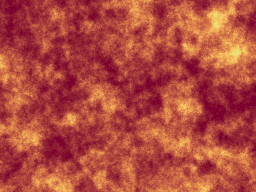}\hfill
  \includegraphics[width=.32\textwidth]{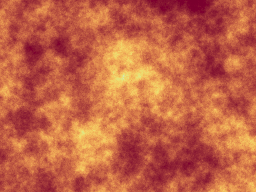}\hfill
  \includegraphics[width=.32\textwidth]{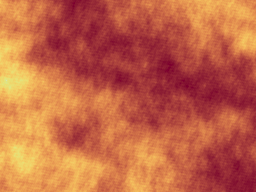}\hspace*{1mm}
  
  \vspace*{1mm}
  
  \hspace*{1mm}\includegraphics[width=.32\textwidth]{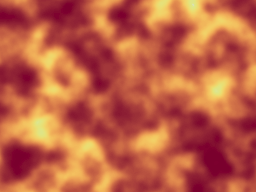}\hfill
  \includegraphics[width=.32\textwidth]{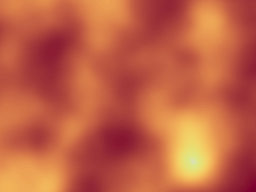}\hfill
  \includegraphics[width=.32\textwidth]{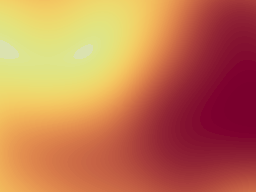}\hspace*{1mm}
\caption{\label{fig:paramcombos}Realizations of the measure $\mu$ for all six parameter combinations in \Secref{sec:performance}. First row: smoothness $s = 0.5$; second row: smoothness $s = 2.5$. The correlation scale $\gamma$ is 0.05, 0.15 and 0.5 (from left to right).}
\end{figure}

The measures \(\nu\) have \(n\) support points generated uniformly at random on $[0,1] \times [0,0.75]$, where we used \(n=250\) and \(n=1000\). We then assigned either mass $1$ or mass $\varrho(x)$ to each point $x$ and normalized to obtain probability measures. We generated 20 independent \(\nu\)-measures of the first kind (unit mass) and computed the optimal transport from each of the $10 \times 6$ \(\mu\)-measures for each of the 6 parameter combinations. We further generated for each of the $10 \times 6$ \(\mu\)-measures $20$ corresponding $\nu$-measure of the second kind (masses from~$\mu$) and computed again the corresponding optimal transports. The stopping criterion for the optimization was an amount of \(\leq 0.05\) of mistransported mass.
\begin{figure}[ht]
\centering
\subfloat[\(n=250,\ s=0.5\)]{\label{fig:runtime_0.5_250}\includegraphics[width=.475\textwidth]{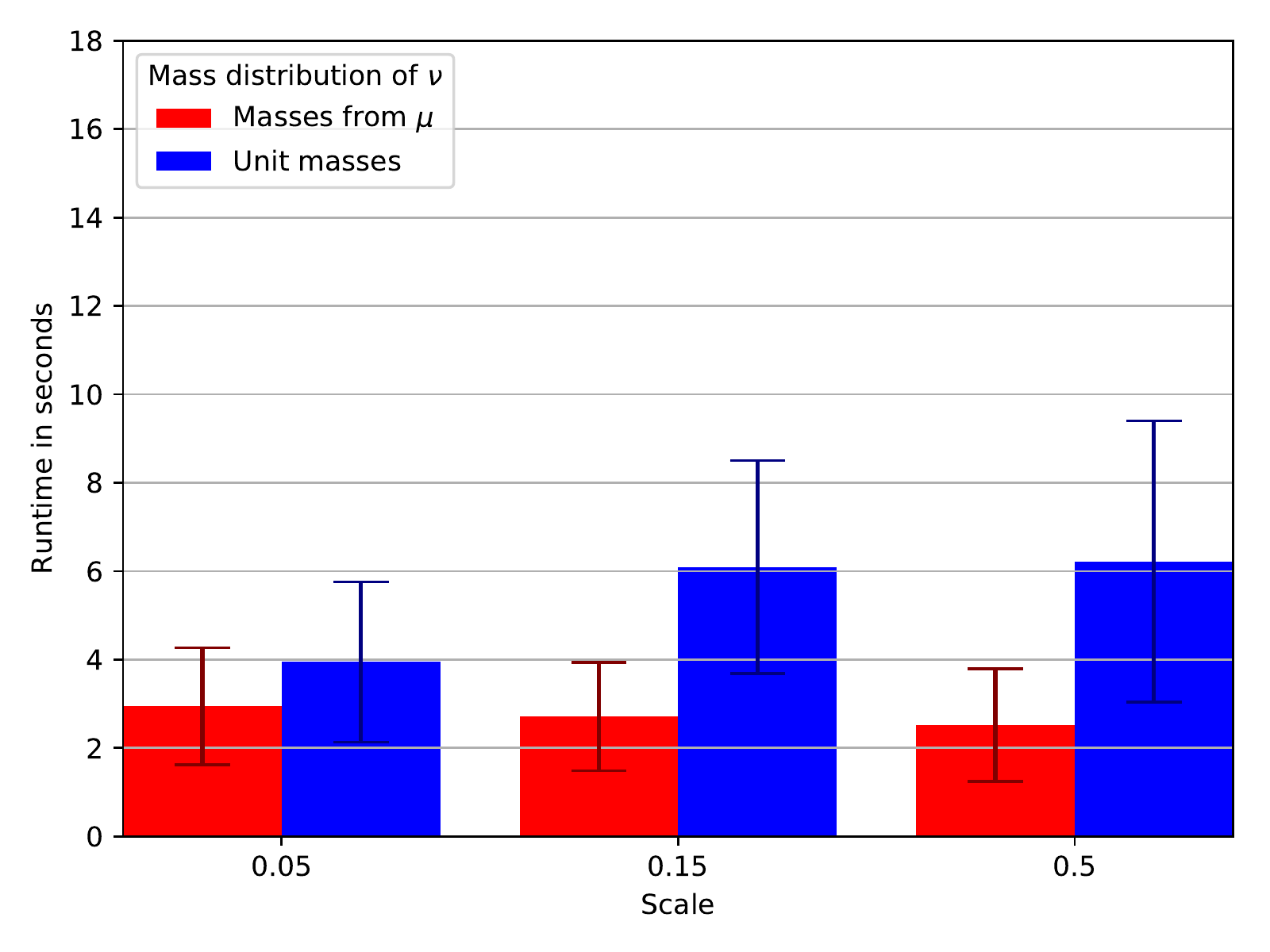}}\hfill
\subfloat[\(n=250,\ s=2.5\)]{\label{fig:runtime_2.5_250}\includegraphics[width=.475\textwidth]{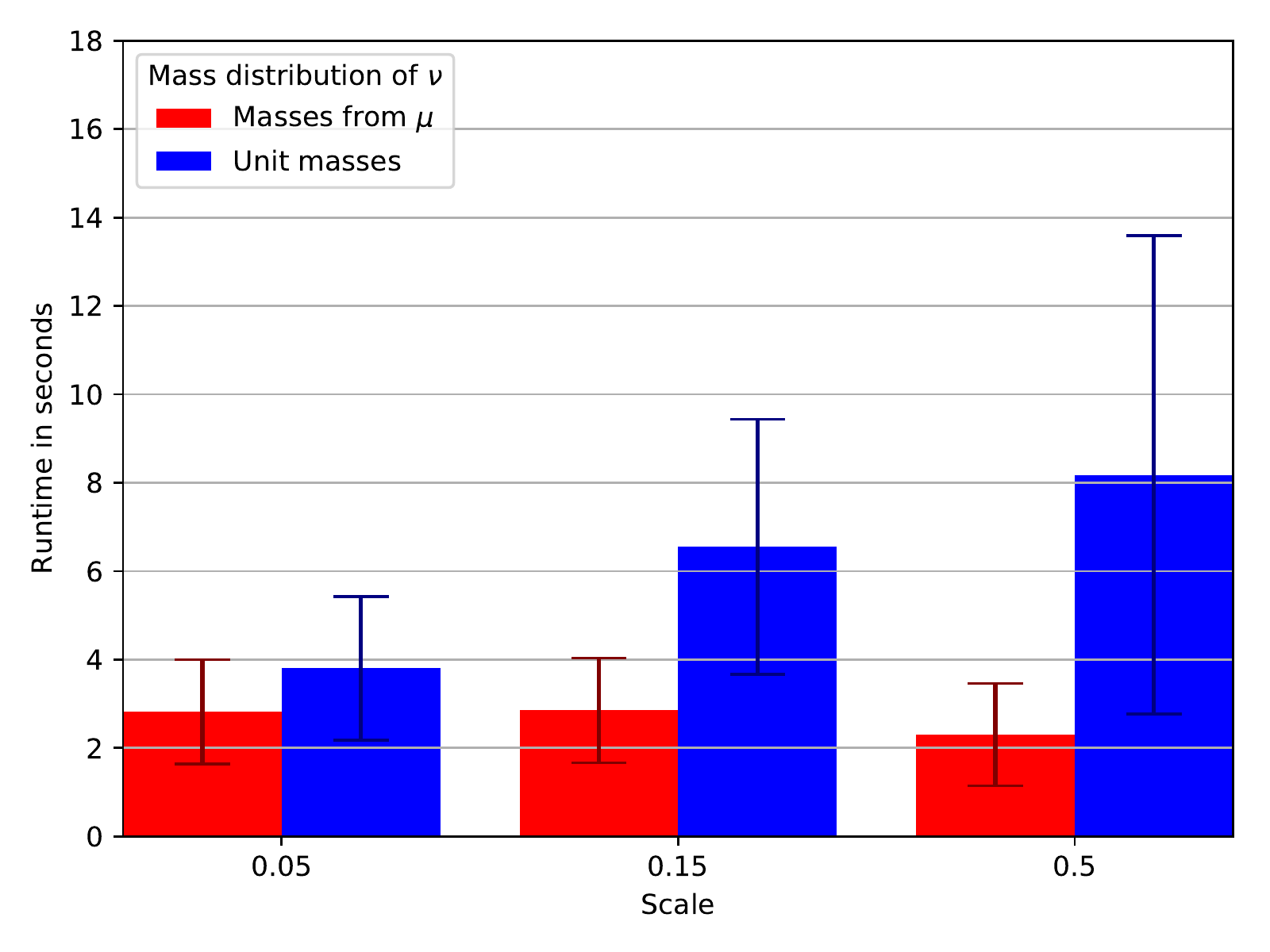}}\
\subfloat[\(n=1000,\ s=0.5\)]{\label{fig:runtime_0.5_1000}\includegraphics[width=.475\textwidth]{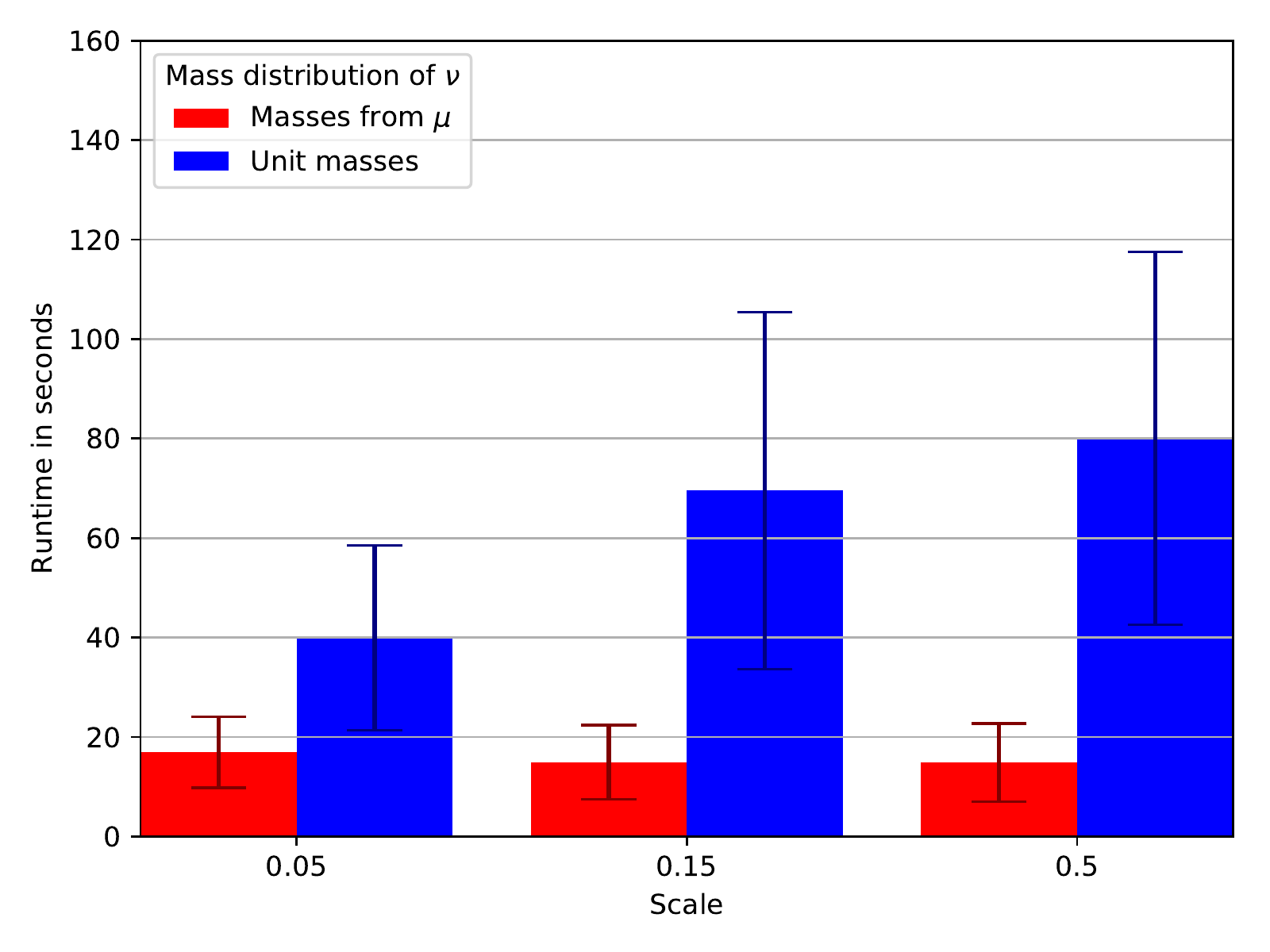}}\hfill
\subfloat[\(n=1000,\ s=2.5\)]{\label{fig:runtime_2.5_1000}\includegraphics[width=.475\textwidth]{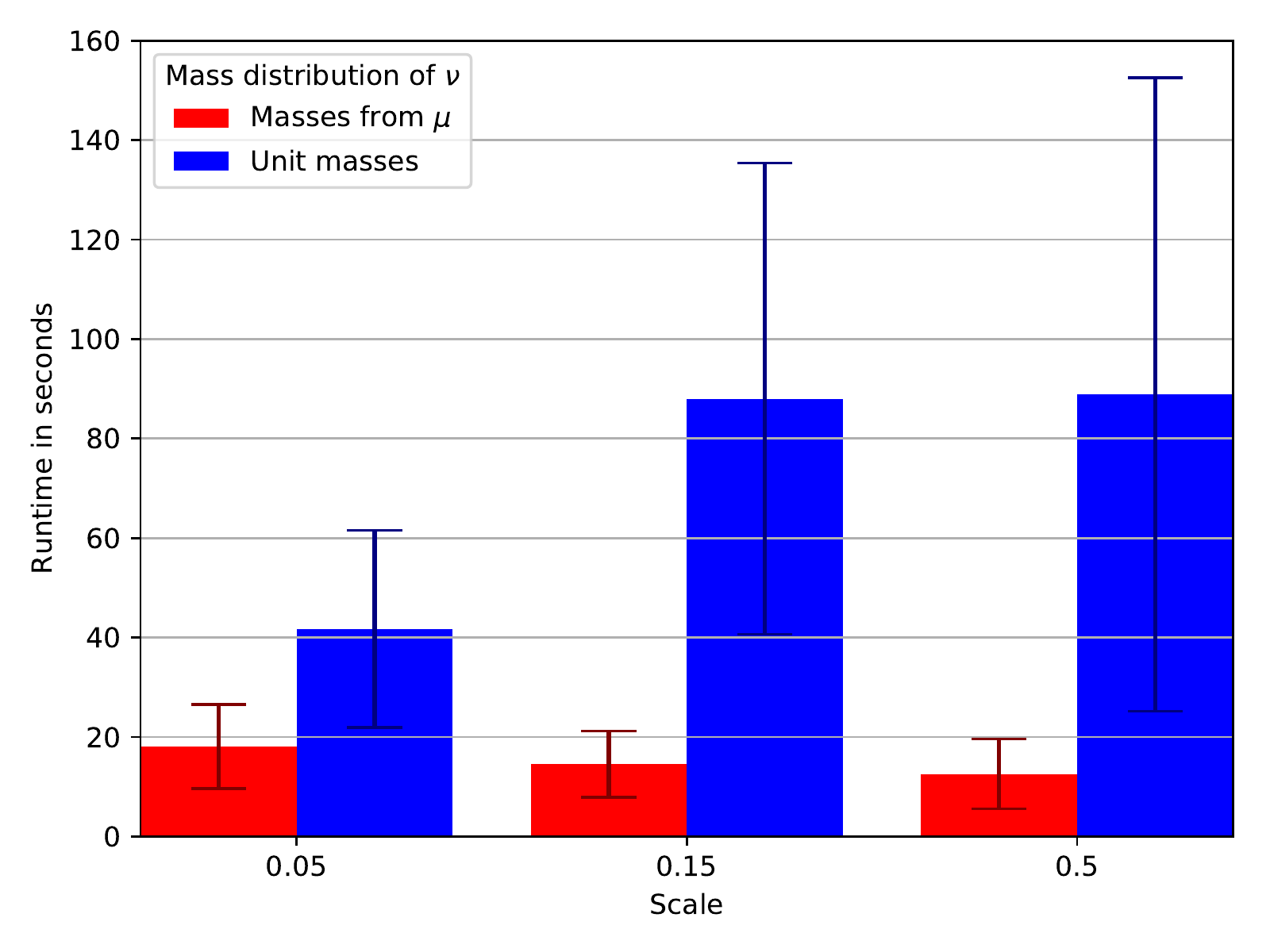}}\
\caption{\label{fig:runtime}Runtimes of the experiments of \Secref{sec:performance}. Bars and lines indicate means and standard deviations over 200 experiments, combining 10 realizations of $\mu$ with 20 realizations of $\nu$. The measures \(\mu\) are based on Gaussian random fields with Mat\'ern covariance function; see \Figref{fig:paramcombos}. The measures \(\nu\) are based on support points picked uniformly at random with unit masses (blue) or masses picked from the corresponding \(\mu\) (red). Rows: \(\nu\) with \(n=250\) versus \(n=1000\) support points. Columns: smoothness parameter \(s=0.5\) versus \(s=2.5\). Note the different scaling.}
\end{figure}

The results for $n=250$ support points of $\nu$ are shown in Figures~\ref{fig:runtime}(a) and \ref{fig:runtime}(b), those for $n=1000$ support points in Figures~\ref{fig:runtime}(c) and \ref{fig:runtime}(d). Each bar shows the mean of the runtimes on one core of a mobile Intel Core i7 across the 200 experiments for the respective parameter combination; the blue bars are for the \(\nu\) measures with uniform masses, the red bars for the measures with masses selected from the corresponding \(\mu\) measure. The lines indicate the standard deviations.

We observe that computation times stay more or less the same between parameter choices (with some sampling variation) if the $\nu$-masses are taken from the corresponding $\mu$-measure. In this case mass can typically be assigned (very) locally, and slightly more so if $\rho$ has fewer local fluctuations (higher $\gamma$ and/or $s$).
This seems a plausible explanation for the relatively small computation times. 

In contrast, if all $\nu$-masses are the same, the computation times are considerably higher and increase substantially with increasing $\gamma$ and somewhat with increasing smoothness. This seems consistent with the hypothesis that the more the optimal transport problem can be solved by assigning mass locally the lower the computation times. For larger scales many of the support points of $\nu$ compete strongly for the assignment of mass and a solution can only be found globally. A lower smoothness may alleviate the problem somewhat, because it creates locally more variation in the available mass. 

We would like to note that to the best of our knowledge the present implementation is the first one for computing the optimal transport in the semi\hyp discrete setting for the case \(p=1\), which means that fair performance comparisons with other algorithms are not easily possible.

\section{Applications}
\label{sec:applications}

We investigate three concrete problem settings in order to better understand the workings and performance of our algorithm as well as to illustrate various theoretical and practical aspects pointed out in the paper.

\subsection{Optimal transport between two normal distributions}
\label{ssec:twonormals}

We consider two bivariate normal distributions $\mu = \mathrm{MVN}_2(a, \sigma^2 \mathrm{I}_2)$ and $\nu = \mathrm{MVN}_2(b, \sigma^2 \mathrm{I}_2)$, where $a = 0.8 \cdot \one$, $b = 2.2 \cdot \one$ and $\sigma^2 = 0.1$, i.e.\ they both have the same spherical covariance matrix such that one distribution is just a displacement of the other. For computations we have truncated both measures to the set $\mcx = [0,3]^2$. By discretization (quantization) a measure $\tnu$ is obtained from $\nu$. We then compute the optimal transport partition and the Wasserstein distances between $\mu$ and $\tnu$ for both $p=1$ and $p=2$. Computations and plots for $p=2$ are obtained with the package \texttt{transport} \cite{transport} for the statistical computing environment {\sf R} \cite{R}. For $p=1$ we use our implementation presented in the previous section.

Note that for the original problem of optimal transport from $\mu$ to $\nu$ the solution is known exactly, so we can use this example to investigate the correct working of our implementation. 
In fact, for any probability measure $\mu'$ on $\RR^d$ and its displacement $\nu' = T_{\#} \mu'$, where $T \colon \RR^2 \to \RR^2, x \mapsto x + (b-a)$ for some vector $b-a \in \RR^d$, it is immediately clear that the translation $T$ induces an optimal transport plan for \eqref{eq:mk} and that $W_p(\mu',\nu') = \norm{b-a}$ for arbitrary $p \geq 1$. This holds because we obtain by Jensen's inequality $(\EE \norm{X-Y}^p)^{1/p} \geq \norm{\EE(X-Y)} = \norm{b-a}$ for $X \sim \mu'$, $Y \sim \nu'$; therefore $W_p(\mu',\nu') \geq \norm{b-a}$ and $T$ is clearly a transport map from $\mu'$ to $\nu'$ that achieves this lower bound. For $p=2$ Theorem~9.4 in \cite{Villani2009} yields that $T$ is the unique optimal transport map and the induced plan $\pi_T$ is the unique optimal transport plan. In the case $p=1$ neither of these objects is unique due to the possibility to rearrange mass transported within the same line segment at no extra cost. 

Discretization was performed by applying the weighted $K$-means algorithm based on the discretization of $\mu$ to a fine grid and an initial configuration of cluster centers drawn independently from distribution $\nu$ and equipped with the corresponding density values of~$\nu$ as weights. The number of cluster centers was kept to $n=300$ for better visibility in the plots below. We write $\tnu = \sum_{i=1}^n \delta_{y_i}$ for the discretized measure. The discretization error can be computed numerically by solving another semi-discrete transport problem, see the third column of Table~\ref{tab:compdist} below. 

The first column of Figure~\ref{fig:transports} depicts the measures $\mu$ and $\tnu$ and the resulting optimal transport partitions for $p=1$ and $p=2$. In the case $p=1$ the nuclei of the weighted Voronoi tessellation are always contained in their cells, whereas for $p=2$ this need not be the case. We therefore indicate the relation by a gray arrow pointing from the centroid of the cell to its nucleus whenever the nucleus is outside the cell. The theory for the case $p=2$, see e.g.\ \cite[Section~2]{Merigot2011}, identifies the tessellation as a Laguerre tessellation (or power diagram), which consists of convex polygons.

The partitions obtained for $p=1$ and $p=2$ look very different, but they both capture optimal transports along the direction $b-a$ very closely. For $p=2$ we clearly see a close approximation of the optimal transport map $T$ introduced above. For $p=1$ we see an approximation of an optimal transport plan $\pi$ that collects the mass for any $y \in \mcy$ somewhere along the way in the direction $b-a$.

The second column of Table~\ref{tab:compdist} gives the Wasserstein distances computed numerically based on these partitions. Both of them are very close to the theoretical value of $\norm{b-a} = \sqrt{2} \cdot 1.4 \approx 1.979899$, and in particular they are well inside the boundaries set by the approximation error.

\begin{table}[ht!]
\centering
\begin{tabular}{|c|lll|lll|}
\hline
& \multicolumn{3}{c|}{\raisebox{0pt}[9pt][0pt]{$\mathrm{MVN}$ vs.~$\mathrm{MVN}$}} & \multicolumn{3}{c|}{$\mathrm{MVN}+\Leb$ vs.~$\mathrm{MVN} + \Leb$} \\[1.5pt]
 & theoretical & computed & discr.\ error & theoretical & computed & discr.\ error \\[1pt]
\hline
p=1 & \raisebox{0pt}[9pt][0pt]{1.979899} & 1.965988 & 0.030962 & 1.979899 & 2.164697 & 0.653370 \\[1.5pt]
p=2 & 1.979899 & 1.965753 & 0.034454 & unknown & 0.827809 & 0.220176 \\
\hline
\end{tabular}
\caption{\label{tab:compdist} Theoretical continuous and computed semidiscrete Wasserstein distances, together with the discretization error.} 
\end{table}

We also investigate the effect of adding a common measure to both $\mu$ and $\nu$: Let $\alpha = \Leb^d\vert_{\mcx}$ and proceed in the same way as above for the measures $\mu' = \mu+\alpha$ and $\nu'=\nu+\alpha$, calling the discretized measure $\tnu'$.
Note that the discretization error (sixth column of Table~\ref{tab:compdist}) is considerably higher, on the one hand due to the fact that the $n=300$ support points of $\tnu'$ have to be spread far wider, on the other hand because the total mass of each measure is $10$ now compared to $1$ before.

\begin{figure}[hp]
\centering
\subfloat[Measures $\mu$ and $\tnu$]{\label{fig:basic}\includegraphics[width=.4\textwidth]{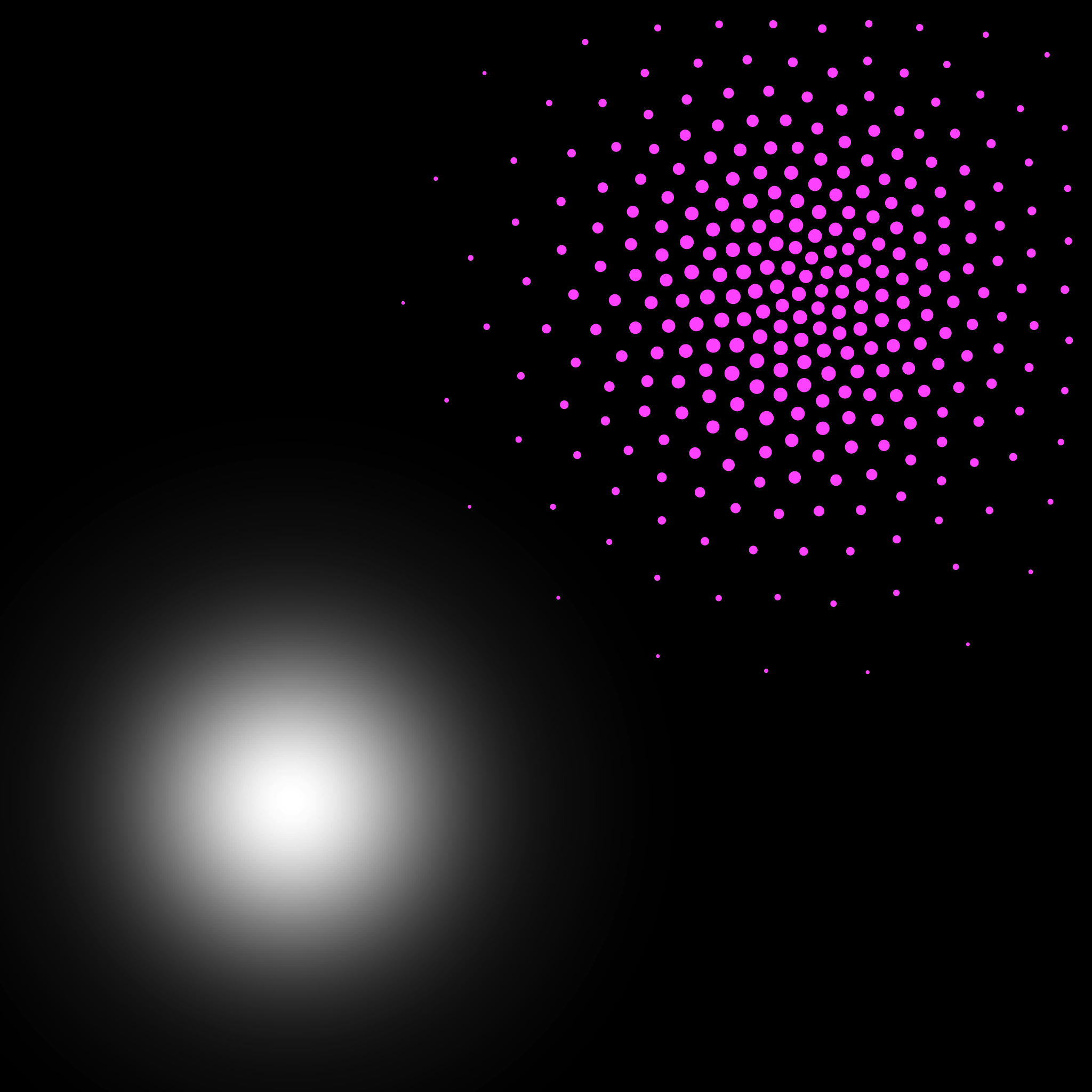}}\hspace*{0.08\textwidth}
\subfloat[Measures $\mu'$ and $\tnu'$]{\label{fig:basic_plus}\includegraphics[width=.4\textwidth]{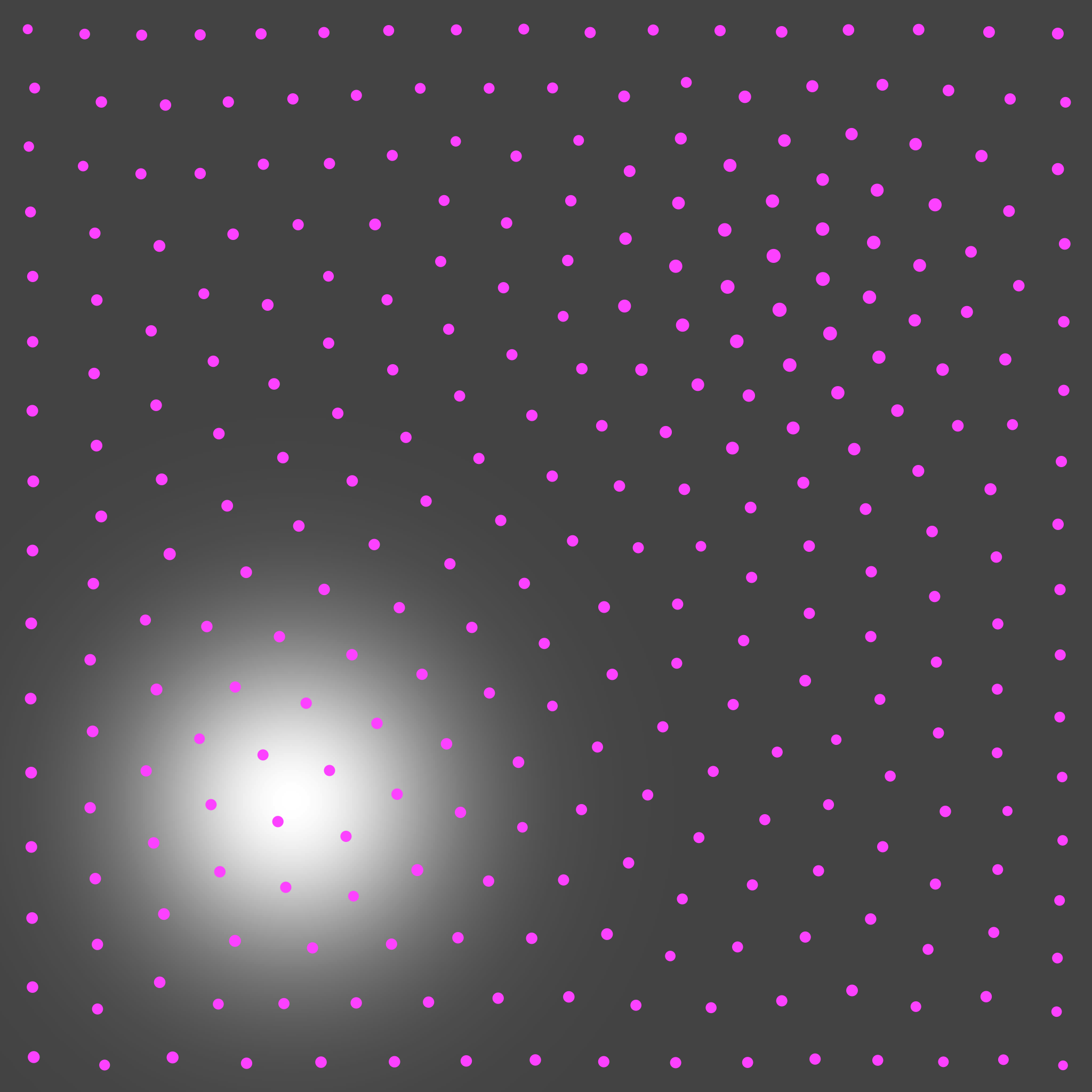}}\
\subfloat[$p=1$]{\label{fig:p1_005}\includegraphics[width=.4\textwidth]{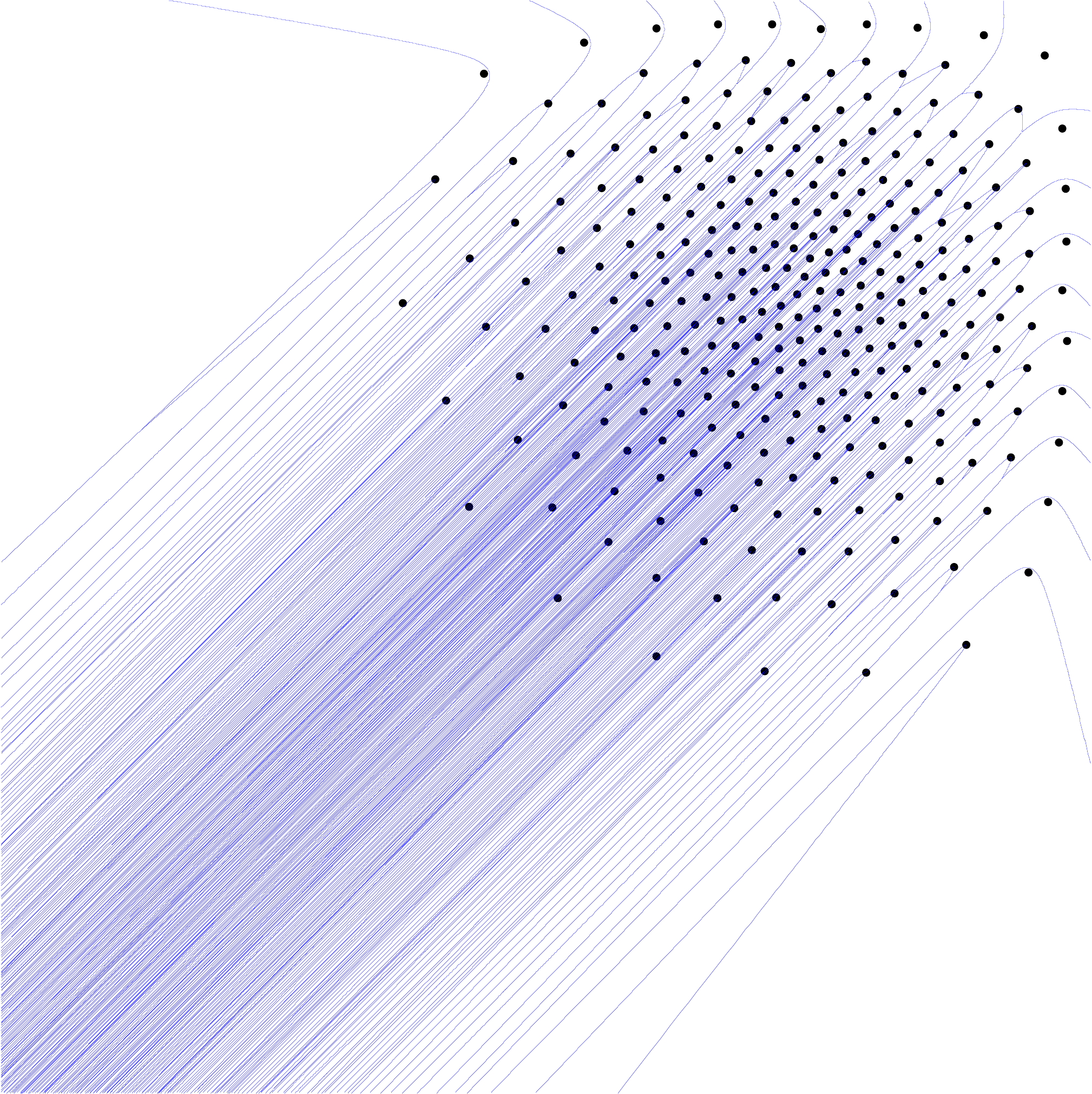}}\hspace*{0.08\textwidth}
\subfloat[$p=1$]{\label{fig:p1_plus_005}\includegraphics[width=.4\textwidth]{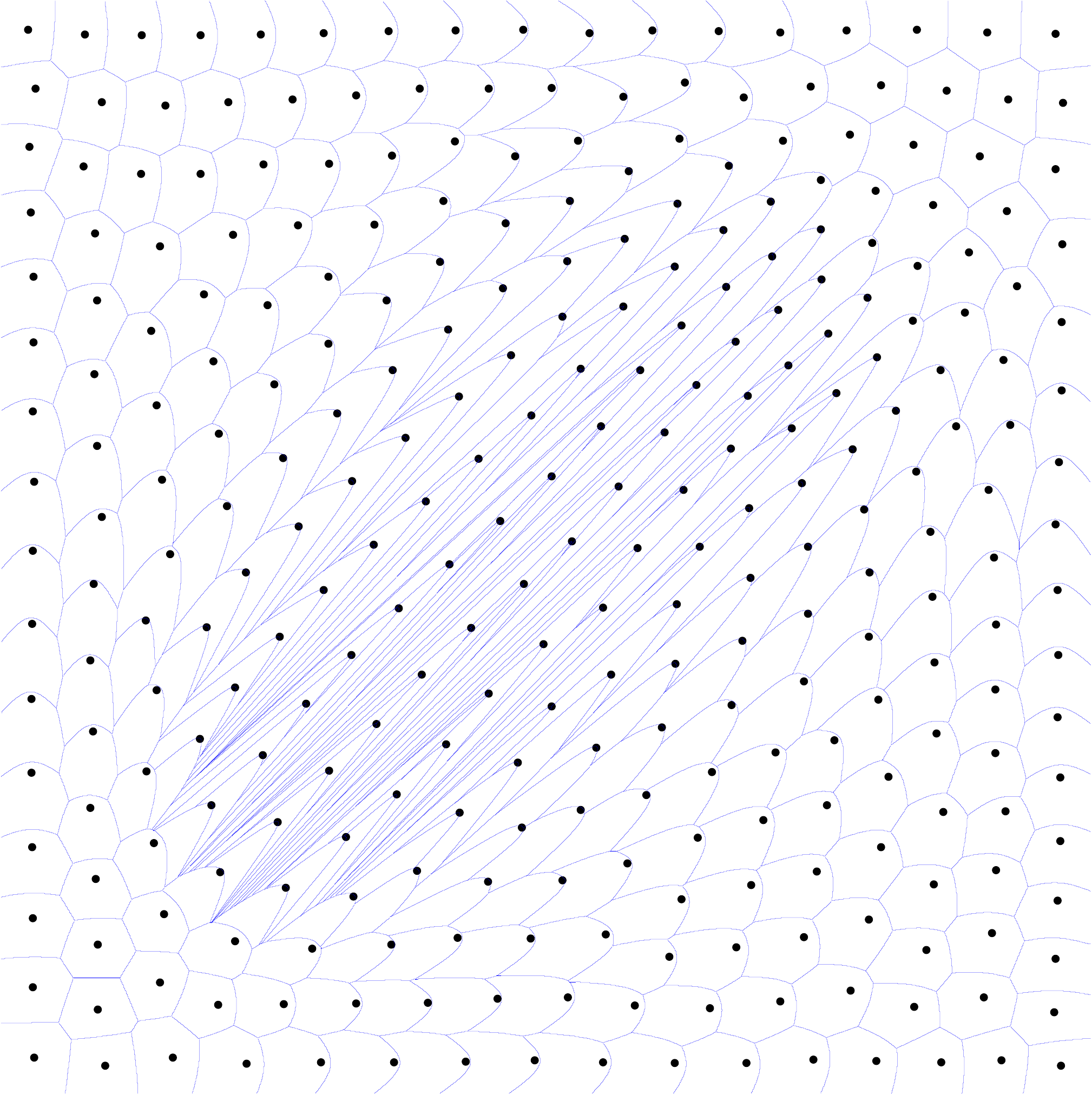}}\
\subfloat[$p=2$]{\label{fig:p2}\includegraphics[width=.4\textwidth]{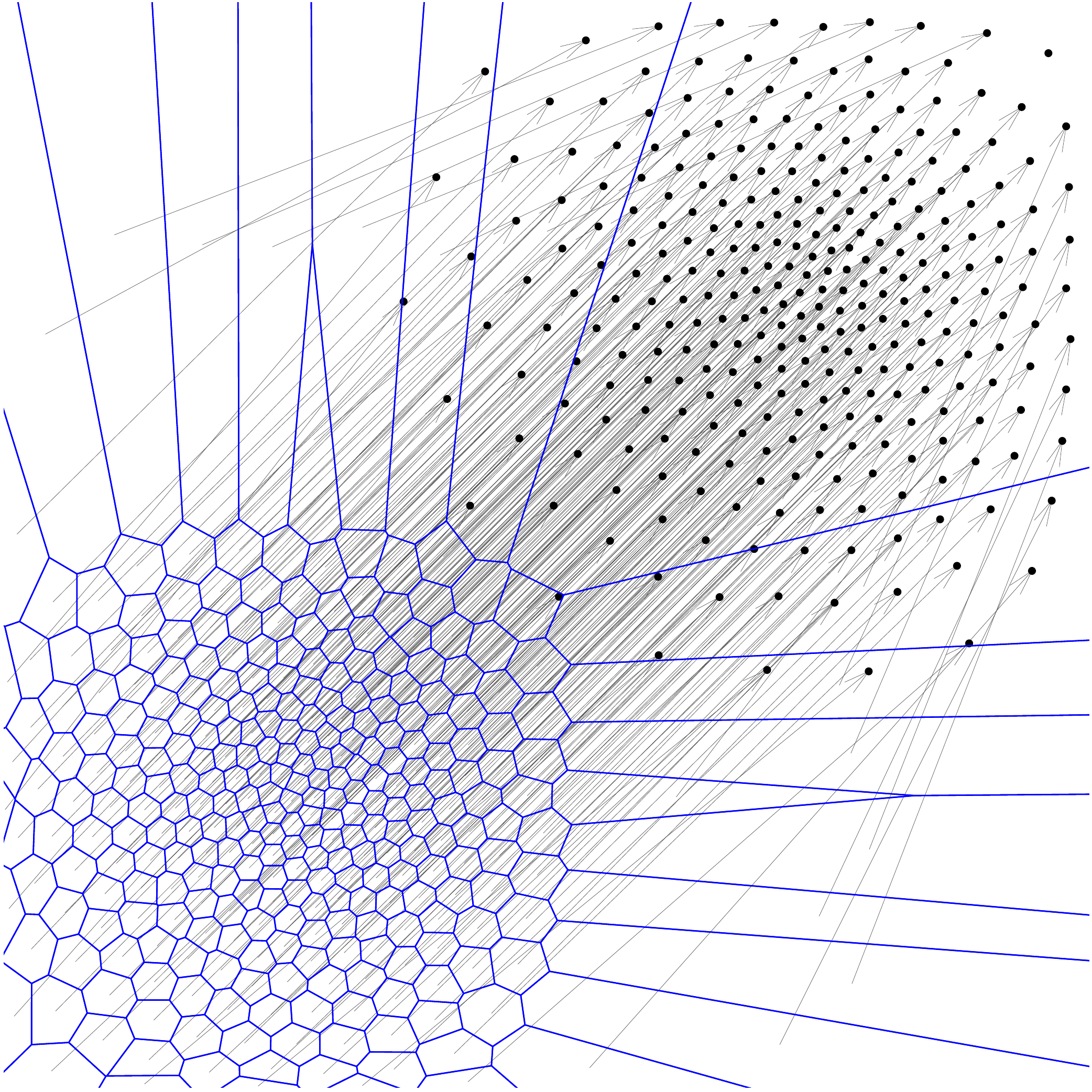}}\hspace*{0.08\textwidth}
\subfloat[$p=2$]{\label{fig:p2_plus}\includegraphics[width=.4\textwidth]{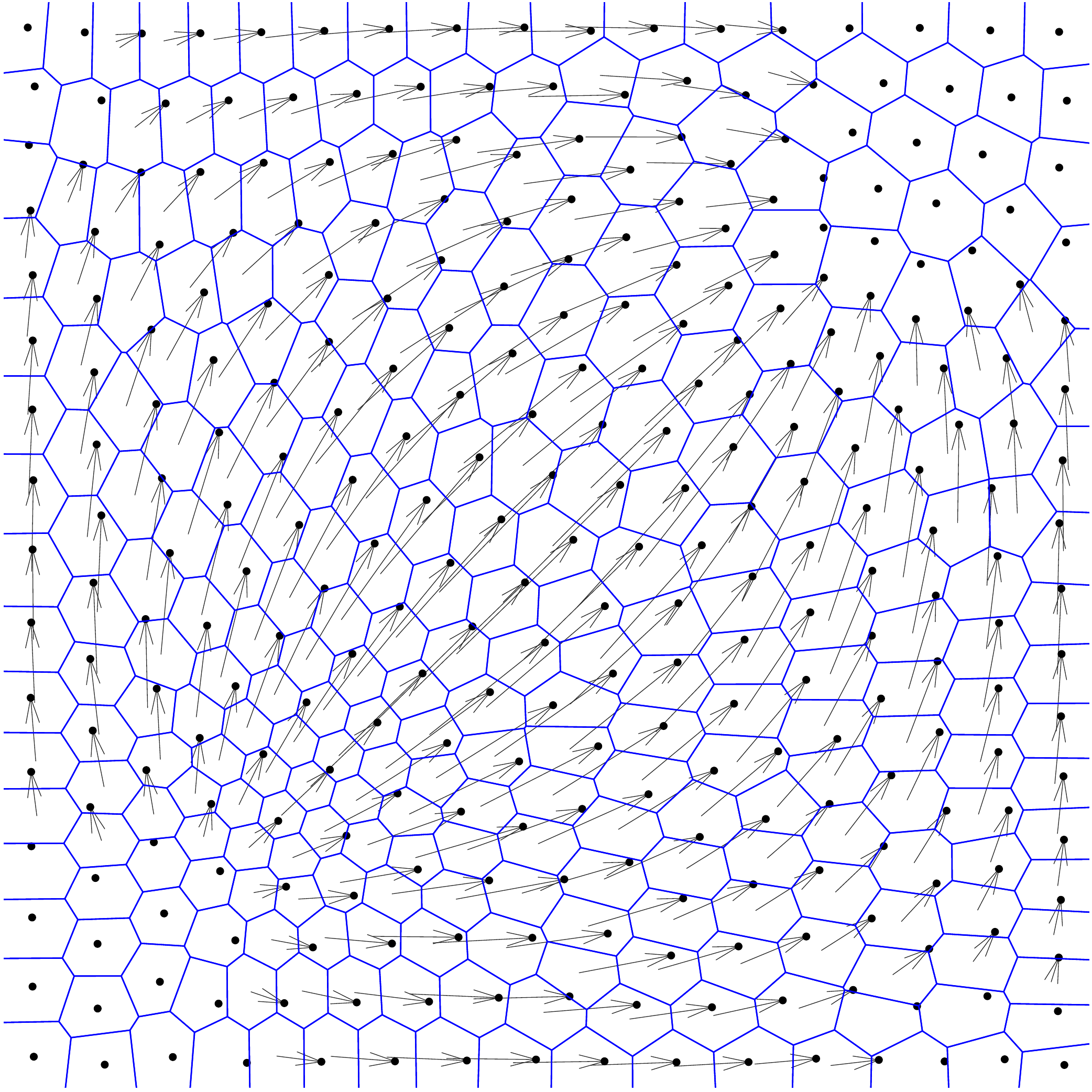}}\
\caption{\label{fig:transports} Left column: Semi-discrete transport between a bivariate normal distribution $\mu$ and a discretized bivariate normal distribution $\tnu$. Right column: Same with Lebesgue measures added to both distributions (before discretization). Panels (a) and (b) illustrate the measures. The densities of the continuous measures $\mu$ and $\mu'$ are displayed as gray level images, the point masses of the discrete measures $\nu$ and $\nu'$ are shown as small discs with areas proportional to the masses placed there. Panels (c) to (f) show the optimal transport partitions.}
\end{figure}

The second column of Figure~\ref{fig:transports} depicts the measures $\mu'$ and $\tnu'$ and the resulting optimal transport partitions for $p=1$ and $p=2$. Both partitions look very different from their counterparts when no $\alpha$ is added. However the partition for $p=1$ clearly approximates a transport plan along the direction of $b-a$ again. Note that the movement of mass is much more local now, meaning the approximated optimal transport plan is not just obtained by keeping measure $\alpha$ in place and moving the remaining measure $\mu$ according to the optimal transport plan $\pi$ approximated in Figure~\ref{fig:transports}(c), but a substantial amount of mass available from $\alpha$ is moved as well. Furthermore, Figure~\ref{fig:transports}(d) gives the impression of a slightly curved movement of mass. We attribute this to a combination of a boundary effect from trimming the Lebesgue measure to $\mcx$ and numerical error based on the coarse discretization and a small amount of mistransported mass.

The computed $W_1$-value for this new example (last column of Table~\ref{tab:compdist}) lies in the vicinity of the theoretical value again if one allows for the rather large discretization error. 

The case $p=2$ exhibits the distinctive curved behavior that goes with the fluid mechanics interpretation discussed in Subsection~\ref{ssec:whyp1}, see also Figure~\ref{fig:circle}. Various of the other points mentioned in Subsection~\ref{ssec:whyp1} can be observed as well, e.g.\ the numerically computed Wasserstein distance 
is much smaller than for $p=1$, which illustrates the lack of invariance and seems plausible in view of the example in Remark~\ref{rem:W2tozero} in the appendix.

\subsection{A practical resource allocation problem}
\label{ssec:realloc}

We revisit the delivery problem mentioned in the introduction. A fast-food delivery service has 32 branches throughout a city area, depicted by the black dots on the map in Figure~\ref{fig:delivery}. For simplicity of representation we assume that most branches have the same fixed production capacity and a few individual ones (marked by an extra circle around the dot) have twice that capacity. We assume further that the expected orders at peak times have a spatial distribution as indicated by the heatmap (where yellow to white means higher number of orders) and a total volume that matches the total capacity of the branches. 
The task of the fast-food chain is to partition the map into 32 delivery zones, matching expected orders in each zone with the capacity of the branches, 
in such a way that the expected cost in form of the travel distance between branch and customer is minimal.
We assume here the Euclidean distance, either because of a street layout that comes close to it, see e.g.~\cite{BoscoeEtAl2012}, or because the deliveries are performed by drones.
The desired partition, computed by our implementation described in Subsection~\ref{ssec:implementation}, is also displayed in Figure~\ref{fig:delivery}. A number of elongated cells in the western and central parts of the city area suggest that future expansions of the fast-food chain should concentrate on the city center in the north.

\begin{figure}[ht!]
	\centering
	\includegraphics[width=0.8\textwidth]{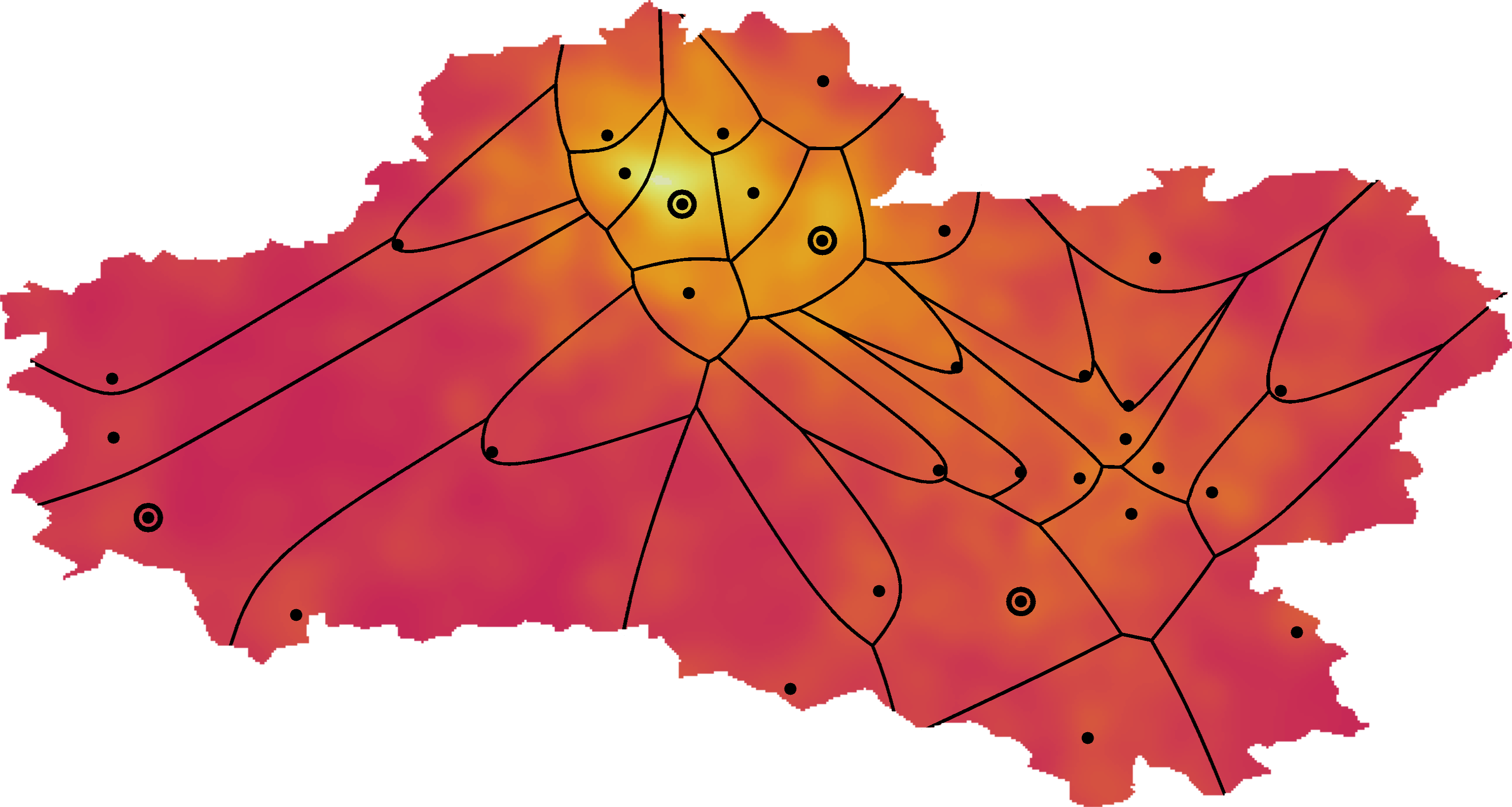}
	\caption{The optimal partition of the city area for the delivery example.}
	\label{fig:delivery}
\end{figure}

\subsection{A visual tool for detecting deviations from a density map}
\label{ssec:vistool}

Very recently, asymptotic theory has been developed that allows, among other things, to test based on the Wasserstein metric $W_p$ whether a sample in $\RR^d$ comes from a given multivariate probability distribution $Q$.
More precisely, assuming independent and identically distributed random vectors $X_1,\ldots,X_n$ with distribution $P$, limiting distributions have been derived for suitable standardizations of $W_p(\frac{1}{n} \sum_{i=1}^n \delta_{X_i}, Q)$ both if $P=Q$ and if $P \neq Q$. Based on an observed value $W_p(\frac{1}{n} \sum_{i=1}^n \delta_{x_i}, Q)$, where $x_1,\ldots,x_n \in \RR^d$, these distributions allow to assign a level of statistical certainty (p-value) to statements of $P=Q$ and $P \neq Q$, respectively.
See~\cite{SommerfeldMunk2018}, which uses general $p \geq 1$, but requires discrete distributions $P$ and $Q$; and~\cite{delBarrioLoubes2018}, which is constraint to $p=2$, but allows for quite general distributions ($P$ and $Q$ not both discrete).

We propose here the optimal transport partition between an absolutely continuous $Q$ and $\frac{1}{n} \sum_{i=1}^n \delta_{x_i}$ as a simple but useful tool for assessing the hypothesis $P=Q$. We refer to this tool as \emph{goodness-of-fit (GOF) partition}. If $d=2$ relevant information may be gained from a simple plot of this partition in a similar way as residual plots are used for assessing the fit of linear models. As a general rule of thumb the partition is consistent with the hypothesis $P=Q$ if it consists of many ``round'' cells that contain their respective $P$-points roughly in their middle. The size of cells may vary according to local densities and there are bound to be some elongated cells due to sampling error (i.e.\ the fact that we can only sample from $P$ and do not know it exactly), but a local accumulation of many elongated cells should give rise to the suspicion that $P=Q$ may be violated in a specific way. Thus GOF partitions provide the data scientist both with a global impression for the plausibility of $P=Q$ and with detailed local information about the nature of potential deviations of $P$ from~$Q$. Of course they are a purely explorative tool and do not give any quantitative guarantees.

We give here an example for illustration. Suppose we have data as given in the left panel of Figure~\ref{fig:vistool_data} and a distribution $Q$ as represented by the heat map in the right panel. Figure~\ref{fig:vistool_partition} shows the optimal transport partition for this situation on the left hand side. The partition indicates that the global fit of the data is quite good. However it also points out some deviations that might be spurious, but might also well be worth further investigation: One is the absence of points close to the two highest peaks in the density map, another one that there are some points too many in the cluster on the very left of the plot. Both of them are quite clearly visible as accumulations of elongated cells.

As an example of a globally bad fit we show in the right panel of Figure~\ref{fig:vistool_partition} the GOF partition when taking as $Q$ the uniform measure on the square. 
\begin{figure}[ht]
\centering
\includegraphics[width=.44\textwidth]{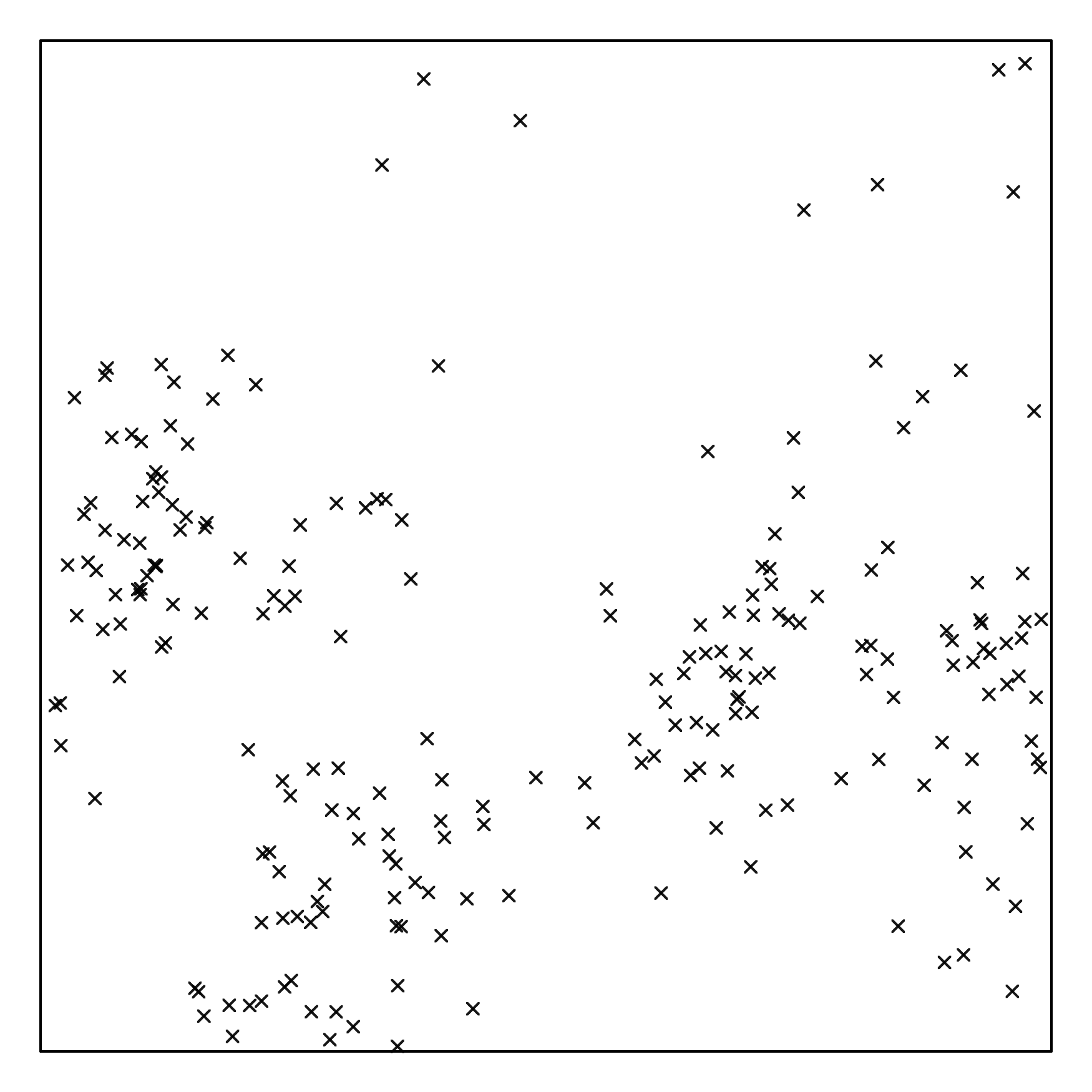}\hspace*{3em}
\raisebox{1.5ex}{\includegraphics[width=.408\textwidth]{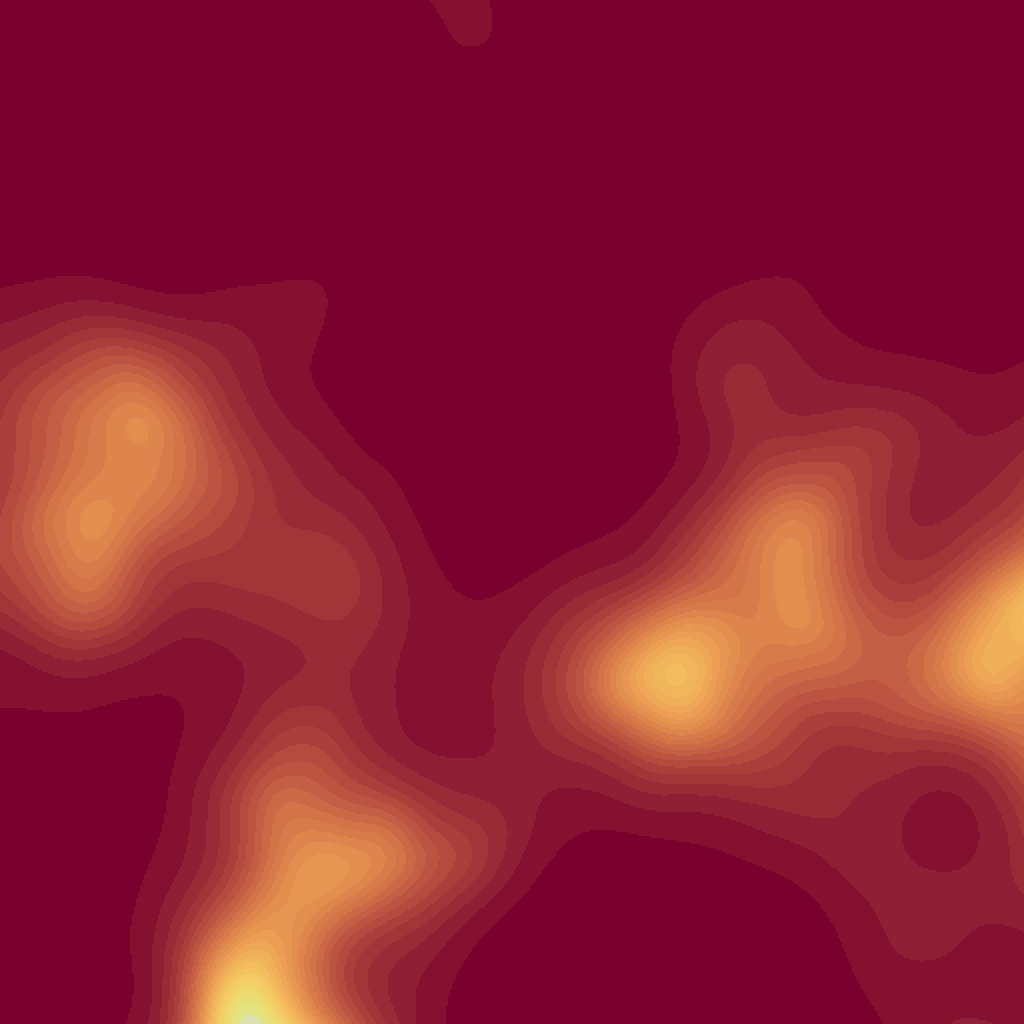}}
\caption{\label{fig:vistool_data} A data example and a continuous density to compare to.}
\end{figure}
\begin{figure}[ht]
\centering
\includegraphics[width=.45\textwidth]{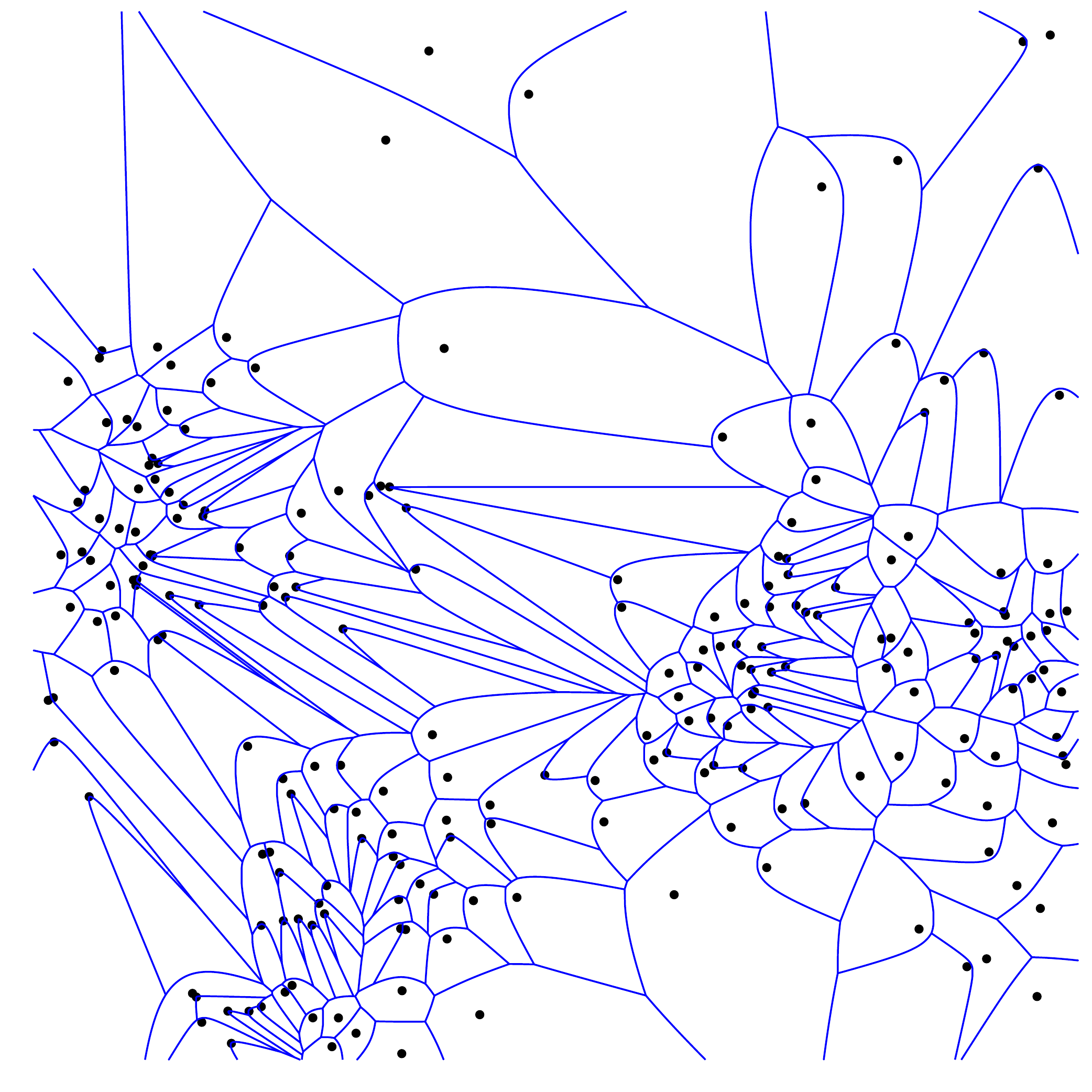}\hspace*{2.5em}
\includegraphics[width=.45\textwidth]{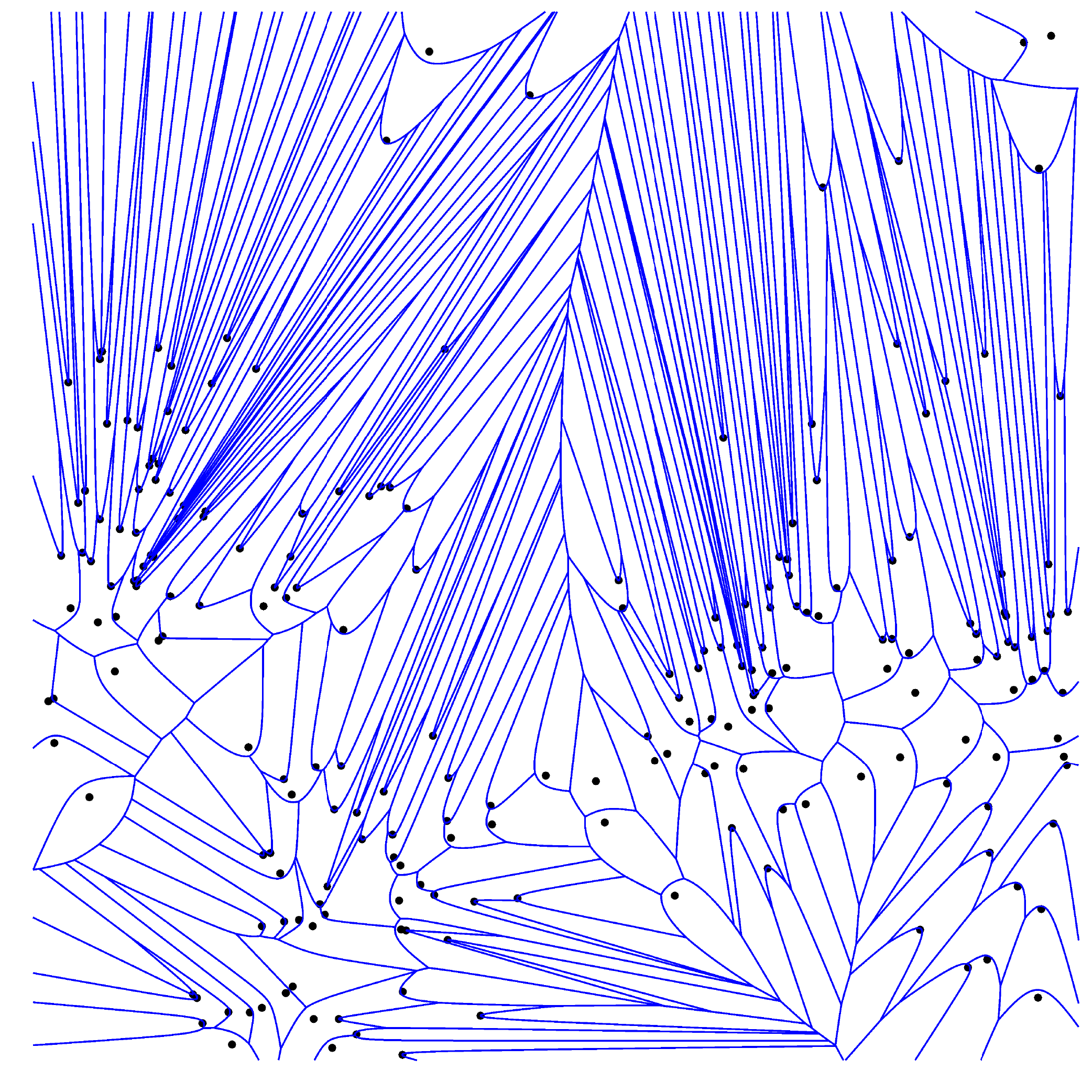}
\caption{\label{fig:vistool_partition} Goodness-of-fit partitions for the data points in the left panel of Figure~\ref{fig:vistool_data} compared with the density in the right panel of Figure~\ref{fig:vistool_data} (left) and compared with the uniform density on the square (right).}
\end{figure}

For larger $d$ direct visual inspection becomes impossible. However, a substantial amount of information may still be extracted, either by considering statistics of the GOF partition in $d$ dimensions that are able to detect local regions of common orientation and high eccentricity of cells, or by applying dimension reduction methods, such as \cite{FlamaryEtAl2018}, before applying the GOF partition.

\section{Discussion and outlook}
\label{sec:discussion}

We have given a comprehensive account on semi-discrete optimal transport for the Euclidean cost function, arguing that there are sometimes good reasons to prefer Euclidean over squared Euclidean cost and showing that for the Euclidean case the semi-discrete setting is particularly nice because we obtain a unique solution to the Monge--Kantorovich problem that is induced by a transport map. We have provided a reasonably fast algorithm that is similar to the AHA-algorithm described in detail in \cite{Merigot2011} but adapted in various aspects to the current situation of $p=1$.

Our algorithm converges towards the optimal partition subject to the convergence conditions for the L-BFGS algorithm; see e.g.~\cite{Nocedal80}. Very loosely, such conditions state that we start in a region around the minimizer where the objective function $\Phi$ shows to some extent quadratic behavior. Similar to the AHA-algorithm in \cite{Merigot2011}, a proof of such conditions is not available. In practice, the algorithm has converged in all the experiments and examples given in the present paper.

There are several avenues for further research, both with regard to improving speed and robustness of the algorithm and for solving more complicated problems where our algorithm may be useful. Some of them are:

\begin{itemize}
  \item As mentioned earlier, it may well be that the choice of our starting value is too simplistic and
  that faster convergence is obtained more often if the sequence $\nu = \nu^{(0)}, \dots, \nu^{(L)}$ of coarsenings is e.g.\ based on the $K$-median algorithm or a similar method. The difficulty lies in finding $\nu^{(l-1)}$ that makes $W_1(\nu^{(l)},\nu^{(l-1)})$ substantially smaller without investing too much time in its computation.

  \item We currently keep the threshold $\eps$ in the stopping criterion of the multiscale approach in Subsection~\ref{ssec:multiscale} fixed. Another alleviation of the computational burden may be obtained by choosing a suitable sequence $\eps_L, \ldots, \eps_0$ of thresholds for the various scales. It seems particularly attractive to use for the threshold at the coarser scale a value $\eps_l > 0$ that is \emph{smaller} than the value $\eps_{l-1}$ at the finer scale, especially for the last step, where $l=1$. The rationale is that at the coarser scale we do not run easily into numerical problems and still reach the stricter $\eps_l$-target efficiently. The obtained weight vector is expected to result in a better starting solution for the finer problem that reaches the more relaxed threshold $\eps_{l-1}$ more quickly than a starting solution stemming from an $\eps_{l-1}$-target at the coarser scale.

	\item The L-BFGS algorithm used for the optimization process may be custom-tailored to our discretization of $\mu$ in order to reach the theoretically maximal numerical precision that the discretization allows. It could e.g.\ use simple gradient descent from the point on where L-BFGS cannot minimize $\Phi$ any further since even in the discretized case the gradient always points in a descending direction.

	\item The approximation of $\mu$ by a fine-grained discrete measure has shown good results. However, as mentioned earlier, higher numerical stability and precision at the expense of complexity and possibly runtime could be obtained by a version of our algorithm that computes the intersections between the Voronoi cells and the pixels of $\mu$ exactly.

  \item Semi-discrete optimal transport may be used as an auxiliary step in a number of algorithms for more complicated problems. The most natural example is a simple alternating scheme for the capacitated location-allocation (or transportation-location) problem; see~\cite{Cooper1972}. Suppose that our fast-food chain from Subsection~\ref{ssec:realloc} has not entered the market yet and would like to open $n$ branches anywhere in the city and divide up the area into delivery zones in such a way that (previously known) capacity constraints of the branches are met and the expected cost in terms of travel distance is minimized again. A natural heuristic algorithm would start with a random placement of $n$ branches and alternate between capacitated allocation of expected orders (the continuous measure $\mu$) using our algorithm described in Section~\ref{sec:algo} and the relocation of branches to the spatial medians of the zones. The latter can be computed by discrete approximation, see e.g.\ \cite{CrouxEtAl2012}, and possibly by continuous techniques, see \cite{FeketeEtAl2005} for a vantage point.
\end{itemize}

\section*{Acknowledgements}
  The authors thank Marcel Klatt for helpful discussions and an anonymous referee for comments that lead to an improvement of the paper.

\bibliographystyle{apalike}      
\bibliography{transport}   

\begin{thebibliography}{}

\bibitem[Altschuler et~al., 2017]{AltschulerEtAl2017}
Altschuler, J., Weed, J., and Rigollet, P. (2017).
\newblock Near-linear time approximation algorithms for optimal transport via
  {S}inkhorn iteration.
\newblock In {\em Proc. NIPS 2017}, pages 1961--1971.

\bibitem[Ambrosio and Pratelli, 2003]{AmbrosioPratelli2003}
Ambrosio, L. and Pratelli, A. (2003).
\newblock Existence and stability results in the {$L^1$} theory of optimal
  transportation.
\newblock In {\em Optimal transportation and applications ({M}artina {F}ranca,
  2001)}, volume 1813 of {\em Lecture Notes in Math.}, pages 123--160.
  Springer, Berlin.

\bibitem[Arjovsky et~al., 2017]{ArjovskyEtAl2017}
Arjovsky, M., Chintala, S., and Bottou, L. (2017).
\newblock Wasserstein generative adversarial networks.
\newblock In {\em Proceedings of the 34th International Conference on Machine
  Learning}, volume~70 of {\em PMLR}, Sydney, Australia.

\bibitem[Armijo, 1966]{Armijo1966}
Armijo, L. (1966).
\newblock Minimization of functions having {L}ipschitz continuous first partial
  derivatives.
\newblock {\em Pacific Journal of Mathematics}, 16(1).

\bibitem[Aurenhammer et~al., 1998]{AHA1998}
Aurenhammer, F., Hoffmann, F., and Aronov, B. (1998).
\newblock Minkowski-type theorems and least-squares clustering.
\newblock {\em Algorithmica}, 20(1):61--76.

\bibitem[Basua et~al., 2014]{MicroscopyImaging2014}
Basua, S., Kolouria, S., and Rohde, G.~K. (2014).
\newblock Detecting and visualizing cell phenotype differences from microscopy
  images using transport-based morphometry.
\newblock {\em PNAS}, 111(9):3448---3453.

\bibitem[Beckmann, 1952]{Beckmann1952}
Beckmann, M. (1952).
\newblock A continuous model of transportation.
\newblock {\em Econometrica}, 20:643--660.

\bibitem[Benamou and Brenier, 2000]{BenamouBrenier2000}
Benamou, J.-D. and Brenier, Y. (2000).
\newblock A computational fluid mechanics solution to the
  {M}onge--{K}antorovich mass transfer problem.
\newblock {\em Numer. Math.}, 84:375--393.

\bibitem[Boscoe et~al., 2012]{BoscoeEtAl2012}
Boscoe, F.~P., Henry, K.~A., and Zdeb, M.~S. (2012).
\newblock A nationwide comparison of driving distance versus straight-line
  distance to hospitals.
\newblock {\em Prof. Geogr.}, 64(2):188--196.

\bibitem[Bourne et~al., 2018]{BourneEtAl2018}
Bourne, D.~P., Schmitzer, B., and Wirth, B. (2018).
\newblock Semi-discrete unbalanced optimal transport and quantization.
\newblock {\em Preprint}.
\newblock \url{https://arxiv.org/abs/1808.01962v1}.

\bibitem[CGAL2015, 2015]{CGAL2015}
CGAL2015 (2015).
\newblock {CGAL}, {C}omputational {G}eometry {A}lgorithms {L}ibrary ({V}ersion
  4.6.1).
\newblock \url{http://www.cgal.org}.

\bibitem[Cooper, 1972]{Cooper1972}
Cooper, L. (1972).
\newblock The transportation-location problem.
\newblock {\em Oper. Res.}, 20(1):94--108.

\bibitem[Courty et~al., 2016]{RemoteSensing2016}
Courty, N., Flamary, R., Tuia, D., and Corpetti, T. (2016).
\newblock Optimal transport for data fusion in remote sensing.
\newblock In {\em 2016 IEEE International Geoscience and Remote Sensing
  Symposium (IGARSS)}, pages 3571--3574.

\bibitem[Croux et~al., 2012]{CrouxEtAl2012}
Croux, C., Filzmoser, P., and Fritz, H. (2012).
\newblock A comparison of algorithms for the multivariate {$L_1$}-median.
\newblock {\em Computational Statistics}, 27(3):393--410.

\bibitem[Cuturi, 2013]{Cuturi2013}
Cuturi, M. (2013).
\newblock Sinkhorn distances: Lightspeed computation of optimal transport.
\newblock In {\em Proc. NIPS 2013}, pages 2292--2300.

\bibitem[del Barrio and Loubes, 2018]{delBarrioLoubes2018}
del Barrio, E. and Loubes, J.-M. (2018).
\newblock Central limit theorems for empirical transportation cost in general
  dimension.
\newblock {\em Preprint}.
\newblock \url{https://arxiv.org/abs/1705.01299v3}.

\bibitem[Fekete et~al., 2005]{FeketeEtAl2005}
Fekete, S.~P., Mitchell, J. S.~B., and Beurer, K. (2005).
\newblock On the continuous {F}ermat-{W}eber problem.
\newblock {\em Operations Research}, 53(1):61--76.

\bibitem[Flamary et~al., 2018]{FlamaryEtAl2018}
Flamary, R., Cuturi, M., Courty, N., and Rakotomamonjy, A. (2018).
\newblock {W}asserstein discriminant analysis.
\newblock {\em Machine Learning}, 107(12):1923--1945.

\bibitem[Gei{\ss} et~al., 2013]{GeissKlein2013}
Gei{\ss}, D., Klein, R., Penninger, R., and Rote, G. (2013).
\newblock Optimally solving a transportation problem using {V}oronoi diagrams.
\newblock {\em Comput. Geom.}, 46(8):1009--1016.

\bibitem[Genevay et~al., 2016]{GenevayEtAl2016}
Genevay, A., Cuturi, M., Peyr{\'e}, G., and Bach, F. (2016).
\newblock Stochastic optimization for large-scale optimal transport.
\newblock In {\em Proc. NIPS 2016}, pages 3432--3440.

\bibitem[Genevay et~al., 2018]{GenevayEtAl2018}
Genevay, A., Peyr{\'e}, G., and Cuturi, M. (2018).
\newblock Learning generative models with {S}inkhorn divergences.
\newblock In {\em Proceedings of the 21st International Conference on
  Artificial Intelligence and Statistics}, volume~84 of {\em PMLR}, Lanzarote,
  Spain.

\bibitem[Gramfort et~al., 2015]{NeuroImaging2015}
Gramfort, A., Peyr{\'e}, G., and Cuturi, M. (2015).
\newblock Fast optimal transport averaging of neuroimaging data.
\newblock In {\em 24th International Conference on Information Processing in
  Medical Imaging (IPMI 2015)}, volume 9123 of {\em Lecture Notes in Computer
  Science}, pages 123--160.

\bibitem[Guo et~al., 2017]{RemoteSensing2017}
Guo, J., Pan, Z., Lei, B., and Ding, C. (2017).
\newblock Automatic color correction for multisource remote sensing images with
  {W}asserstein {CNN}.
\newblock {\em Rem. Sens.}, 9(5):1--16 (electronic).

\bibitem[Hartmann, 2016]{Hartmann2016}
Hartmann, V. (2016).
\newblock A geometry-based approach for solving the transportation problem with
  {E}uclidean cost.
\newblock Bachelor's thesis, Institute of Mathematical Stochastics, University
  of G\"ottingen.
\newblock \url{https://arxiv.org/abs/1706.07403}.

\bibitem[Kantorovich, 1942]{Kantorovich1942}
Kantorovich, L. (1942).
\newblock On the translocation of masses.
\newblock {\em C. R. (Doklady) Acad. Sci. URSS (N.S.)}, 37:199--201.

\bibitem[Karavelas and Yvinec, 2002]{KaravelasYvinec2002}
Karavelas, M.~I. and Yvinec, M. (2002).
\newblock Dynamic additively weighted {V}oronoi diagrams in 2{D}.
\newblock In {\em Algorithms -- ESA 2002}, pages 586--598. Springer.

\bibitem[Kitagawa et~al., 2018]{KitagawaEtAl2017}
Kitagawa, J., M\'erigot, Q., and Thibert, B. (2018).
\newblock Convergence of a {N}ewton algorithm for semi-discrete optimal
  transport.
\newblock {\em J. Eur. Math. Soc., to appear}.
\newblock \url{http://arxiv.org/abs/1603.05579v2}.

\bibitem[Klatt and Munk, 2018]{KlattMunk2018}
Klatt, M. and Munk, A. (2018).
\newblock Limit distributions for empirical regularized optimal transport and
  applications.
\newblock {\em In preparation}.

\bibitem[Luenberger and Ye, 2008]{LuenbergerYe2008}
Luenberger, D.~G. and Ye, Y. (2008).
\newblock {\em Linear and nonlinear programming}.
\newblock International Series in Operations Research \& Management Science,
  116. Springer, New York, third edition.

\bibitem[McCann, 1995]{McCann1995}
McCann, R.~J. (1995).
\newblock Existence and uniqueness of monotone measure-preserving maps.
\newblock {\em Duke Math. J.}, 80(2):309--323.

\bibitem[M\'erigot, 2011]{Merigot2011}
M\'erigot, Q. (2011).
\newblock A multiscale approach to optimal transport.
\newblock {\em Comput. Graph. Forum}, 30(5):1583--1592.

\bibitem[Monge, 1781]{Monge1781}
Monge, G. (1781).
\newblock M\'emoire sur la th\'eorie des d\'eblais et des remblais.
\newblock In {\em Histoire de l'Acad\'emie Royale des Sciences de Paris, avec
  les M\'emoires de Math\'ematique et de Physique pour la m\^eme ann\'ee,},
  pages 666--704.

\bibitem[Nicolas, 2016]{Papadakis2016}
Nicolas, P. (2016).
\newblock Optimal transport for image processing.
\newblock Habilitation thesis, Signal and Image Processing, Universit\'e de
  Bordeaux.
\newblock \url{https://hal.archives-ouvertes.fr/tel-01246096v6}.

\bibitem[Nocedal, 1980]{Nocedal80}
Nocedal, J. (1980).
\newblock Updating quasi-{N}ewton matrices with limited storage.
\newblock {\em Mathematics of computation}, 35(151):773--782.

\bibitem[Okazaki and Nocedal, 2010]{libLBFGS2015}
Okazaki, N. and Nocedal, J. (2010).
\newblock {libLBFGS} ({V}ersion 1.10).
\newblock \url{http://www.chokkan.org/software/liblbfgs/}.

\bibitem[Peyr{\'e} and Cuturi, 2018]{PeyreCuturi2018}
Peyr{\'e}, G. and Cuturi, M. (2018).
\newblock {\em Computational Optimal Transport}.
\newblock now Publishers.
\newblock \url{https://arxiv.org/abs/1803.00567}.

\bibitem[Pratelli, 2007]{Pratelli2007}
Pratelli, A. (2007).
\newblock On the equality between {M}onge's infimum and {K}antorovich's minimum
  in optimal mass transportation.
\newblock {\em Ann. Inst. H. Poincar\'e Probab. Statist.}, 43(1):1--13.

\bibitem[{R Core Team}, 2017]{R}
{R Core Team} (2017).
\newblock {\em R: A Language and Environment for Statistical Computing}.
\newblock R Foundation for Statistical Computing, Vienna, Austria.
\newblock Version 3.3.0. \, \url{https://www.R-project.org/}.

\bibitem[Rippl et~al., 2016]{RipplEtAl2016}
Rippl, T., Munk, A., and Sturm, A. (2016).
\newblock Limit laws of the empirical {W}asserstein distance: {G}aussian
  distributions.
\newblock {\em J. Multivar. Anal.}, 151:90--109.

\bibitem[Santambrogio, 2015]{Santambrogio2015}
Santambrogio, F. (2015).
\newblock {\em Optimal transport for applied mathematicians}, volume~87 of {\em
  Progress in Nonlinear Differential Equations and their Applications}.
\newblock Birkh\"auser/Springer, Cham.

\bibitem[Schmitz et~al., 2018]{SchmitzEtAl2018}
Schmitz, M.~A., Heitz, M., Bonneel, N., Ngol{\`e}, F., Coeurjolly, D., Cuturi,
  M., Peyr{\'e}, G., and Starck, J.-L. (2018).
\newblock Wasserstein dictionary learning: Optimal transport-based unsupervised
  nonlinear dictionary learning.
\newblock {\em SIAM Journal on Imaging Sciences}, 11(1):643--678.

\bibitem[Schmitzer, 2016a]{Schmitzer2016}
Schmitzer, B. (2016a).
\newblock A sparse multiscale algorithm for dense optimal transport.
\newblock {\em J Math Imaging Vis}, 56(2):238--259.

\bibitem[Schmitzer, 2016b]{Schmitzer2016entropy}
Schmitzer, B. (2016b).
\newblock Stabilized sparse scaling algorithms for entropy regularized
  transport problems.
\newblock {\em Preprint}.
\newblock \url{https://arxiv.org/abs/1610.06519v1}.

\bibitem[Schmitzer and Wirth, 2018]{SchmitzerWirth2017}
Schmitzer, B. and Wirth, B. (2018).
\newblock A framework for {W}asserstein-1-type metrics.
\newblock {\em Preprint}.
\newblock \url{http://arxiv.org/abs/1701.01945v2}.

\bibitem[Schrieber et~al., 2017]{SchrieberEtAl2017}
Schrieber, J., Schuhmacher, D., and Gottschlich, C. (2017).
\newblock {DOT}mark --- a benchmark for discrete optimal transport.
\newblock {\em IEEE Access}, 5.

\bibitem[Schuhmacher et~al., 2018]{transport}
Schuhmacher, D., B{\"a}hre, B., Gottschlich, C., Heinemann, F., and Schmitzer,
  B. (2018).
\newblock {\em transport: Optimal Transport in Various Forms}.
\newblock R package version 0.9-4. \,
  \url{https://cran.r-project.org/package=transport}.

\bibitem[Sherali and Nordai, 1988]{SheraliNordai1988}
Sherali, H.~D. and Nordai, F.~L. (1988).
\newblock {NP}-hard, capacitated, balanced p-median problems on a chain graph
  with a continuum of link demands.
\newblock {\em Math. Oper. Res.}, 13(1):32--49.

\bibitem[Solomon et~al., 2015]{SolomonEtAl2015}
Solomon, J., de~Goes, F., Peyr{\'e}, G., Cuturi, M., Butscher, A., Nguyen, A.,
  Du, T., and Guibas, L. (2015).
\newblock Convolutional {W}asserstein distances: Efficient optimal
  transportation on geometric domains.
\newblock {\em ACM Trans. Graph.}, 34(4):66:1--66:11.

\bibitem[Solomon et~al., 2014]{SolomonEtAl2014}
Solomon, J., Rustamov, R., Guibas, L., and Butscher, A. (2014).
\newblock Convolutional {W}asserstein distances: Efficient optimal
  transportation on geometric domains.
\newblock {\em ACM Trans. Graph.}, 33(4):67:1--67:12.

\bibitem[Sommerfeld and Munk, 2018]{SommerfeldMunk2018}
Sommerfeld, M. and Munk, A. (2018).
\newblock Inference for empirical {W}asserstein distances on finite spaces.
\newblock {\em Journal of the Royal Statistical Society: Series B (Statistical
  Methodology)}, 80(1):219--238.

\bibitem[Villani, 2009]{Villani2009}
Villani, C. (2009).
\newblock {\em Optimal transport, old and new}, volume 338 of {\em Grundlehren
  der Mathematischen Wissenschaften [Fundamental Principles of Mathematical
  Sciences]}.
\newblock Springer-Verlag, Berlin.

\bibitem[Wolansky, 2015]{Wolansky2015}
Wolansky, G. (2015).
\newblock Semi-discrete approximation of optimal mass transport.
\newblock {\em Preprint}.
\newblock \url{http://arxiv.org/abs/1502.04309v1}.

\bibitem[Wolfe, 1969]{Wolfe1969}
Wolfe, P. (1969).
\newblock Convergence conditions for ascent methods.
\newblock {\em SIAM Rev.}, 11:226--235.

\bibitem[Wolfe, 1971]{Wolfe1971}
Wolfe, P. (1971).
\newblock Convergence conditions for ascent methods. {II}. {S}ome corrections.
\newblock {\em SIAM Rev.}, 13:185--188.

\end{thebibliography}

\vspace{8mm}

\noindent
\rule{0.98\textwidth}{0.3pt}
\vspace*{3mm}

{\footnotesize

\noindent
\parbox[t]{80mm}{Valentin Hartmann \\
              IC IINFCOM DLAB \\
              EPFL \\
	      Station 14 \\
              1015 Lausanne, Switzerland \\[2mm]
               \emph{E-mail:} valentin.hartmann@epfl.ch
             }
\parbox[t]{80mm}{Dominic Schuhmacher \\
              Institute for Mathematical Stochastics \\
              University of Goettingen \\
              Goldschmidtstr.\ 7 \\
              37077 Göttingen, Germany \\[2mm]
              \emph{E-mail:} schuhmacher@math.uni-goettingen.de
            }
          }
\vspace*{3mm}

\section*{Appendix: Formulae for affine transformations of measures}

We have the following relations when adding a common measure or multiplying by a common nonnegative scalar. The proof easily extends to a complete separable metric space
instead of $\RR^d$ equipped with the Euclidean metric. 
\begin{lemma} \label{lem:affine}
  Let $\mu,\nu,\alpha$ be finite measures on $\RR^d$ satisfying $\mu(\RR^d) = \nu(\RR^d)$. For $p \geq 1$ and $c > 0$, we have
  \begin{align}
    W_p(\alpha + \mu, \alpha + \nu) &\leq W_p(\mu, \nu), \label{eq:addp} \\
    W_1(\alpha + \mu, \alpha + \nu) &= W_1(\mu, \nu), \label{eq:add1} \\    
    W_p(c \mu, c \nu) &= c^{1/p} \hbit W_p(\mu, \nu),       \label{eq:scmult}
  \end{align}
where we assume for $\eqref{eq:add1}$ that $W_1(\mu, \nu) < \infty$.
\end{lemma}
\begin{proof}
  Write $\Delta = \Set{(x,x) \given x \in \RR^d}$. Denote by $\alpha_{\Delta}$ the push-forward of $\alpha$ under the map $[\RR^d \to \RR^d \times \RR^d, x \mapsto (x,x)]$.
  Let $\pi_*$ be an optimal transport plan for the computation of $W_p(\mu, \nu)$. Then $\pi_*+\alpha_{\Delta}$ is a feasible transport plan for $W_p(\alpha + \mu, \alpha + \nu)$ that generates the same total cost as $\pi_*$. Thus 
  \begin{equation*}
    W_p(\alpha + \mu, \alpha + \nu) \leq W_p(\mu, \nu).
  \end{equation*}
  Likewise $c \pi_*$ is a feasible transport plan for $W_p(c \mu, c \nu)$
  that generates $c^{1/p}$ times the cost of $\pi_*$ for the integral in~\eqref{eq:mk}. Thus
  \begin{equation*}
    W_p(c \mu, c \nu) \leq c^{1/p} \hbit W_p(\mu, \nu).
  \end{equation*}
Replacing $c$ by $1/c$, as well as $\mu$ by $c\mu$ and $\nu$ by $c\nu$, we obtain~\eqref{eq:scmult}.

It remains to show $W_1(\alpha + \mu, \alpha + \nu) \geq W_1(\mu, \nu)$. For this we use that a transport plan $\pi$ between $\mu$ and $\nu$ is optimal if and only if it is cyclical monotone, meaning that for all $N \in \NN$ and all $(x_1,y_1), \ldots, (x_N,y_N) \in \supp(\pi)$, we have
\begin{equation*}
  \sum_{i=1}^N \norm{x_i-y_i} \leq \sum_{i=1}^N \norm{x_i-y_{i+1}},
\end{equation*}
where $y_{N+1} = y_1$; see \cite[Theorem~5.10(ii) and Definition~5.1]{Villani2009}.

Letting $\pi_*$ be an optimal transport plan for the computation of $W_1(\mu, \nu)$, we show optimality of $\pi_* + \alpha_{\Delta}$ for the computation of $W_1(\mu+\alpha, \nu+\alpha)$. We know that $\pi_*$ is cyclical monotone. Let $N \in \NN$ and $(x_1,y_1), \ldots, (x_N,y_N) \in \supp(\pi_* + \alpha_{\Delta}) \subset \supp(\pi_*) \cup \Delta$. Denote by $1 \leq i_1 < \ldots < i_k \leq N$, where $k \in \{0,\ldots,N\}$, the indices of all pairs with $x_{i_j} \neq y_{i_j}$, and hence $(x_{i_j},y_{i_j}) \in \supp(\pi_*)$. By the cyclical monotonicity of $\pi_*$ (writing $i_{k+1}=i_1$) and the triangle inequality, we obtain   
\begin{equation*}
  \sum_{i=1}^N \norm{x_i-y_i} = \sum_{j=1}^k \norm{x_{i_j}-y_{i_j}}
  \leq \sum_{j=1}^k \norm{x_{i_j}-y_{i_{j+1}}} \leq \sum_{i=1}^N \norm{x_{i}-y_{i+1}}.
\end{equation*}
Thus $\pi_* + \alpha_{\Delta}$ is cyclical monotone and since it is a feasible transport plan between $\mu+\alpha$ and $\nu+\alpha$, it is optimal for the computation of $W_1(\mu+\alpha, \nu+\alpha)$, which concludes the proof. 
\end{proof}

\begin{remark} \label{rem:W2tozero}
  Equation~\eqref{eq:add1} is not generally true for any $p > 1$. To see this consider
  the case $d=1$, $\mu=\delta_0$, $\nu=\delta_1$ and $\alpha = b \hbit \Leb \vert_{[0,1]}$, where $b \geq 1$.
Clearly $W_p(\mu,\nu)=1$ for all $p \geq 1$. Denote by $F$ and $G$ the cumulative distribution functions (CDFs) of $\mu+\alpha$ and $\nu+\alpha$, respectively, i.e.\ $F(x) = \mu((-\infty,x])$ and $G(x) = \nu((-\infty,x])$ for all $x \in \RR$. Thus
\begin{equation*}
\begin{cases}
  F(x)=G(x)=0 &\text{if $x<0$}, \\
  F(x) = 1+bx,\ G(x) = bx &\text{if $x \in [0,1)$}, \\
  F(x)=G(x)=b+1 &\text{if $x \geq 1$}.
\end{cases}
\end{equation*}
We then even obtain
\begin{equation*}
\begin{split}
  W_p^{\hspace*{0.5pt} \raisebox{1.5pt}{$\scriptstyle p$}}(\alpha + \mu, \alpha + \nu) &= \int_{0}^{b+1} \abs{F^{-1}(t)-G^{-1}(t)}^p \; dt \\
  &= 2 \int_0^1 \frac{t^p}{b^p} \; dt + \frac{1}{b^p} (b-1) 
  = \frac{1}{b^p} \Bigl(b-1 + \frac{2}{p+1} \Bigr) \longrightarrow 0 
\end{split}
\end{equation*}
as $b \to \infty$ if $p>1$. For the first equality we used the representation of $W_p$ in terms of (generalized) inverses of their CDFs; see Equation~(2) in \cite{RipplEtAl2016} and the references given there and note that the generalization from the result for probability measures is immediate by~\eqref{eq:scmult}. 
\end{remark}

\end{document}